\documentclass{amsart}

\makeatletter
\@namedef{subjclassname@2020}{%
  \textup{2020} Mathematics Subject Classification}
  
\makeatother 

\usepackage{amsthm,amssymb,amsfonts,latexsym,mathtools,thmtools,amsmath,lscape}
\usepackage[T1]{fontenc}
\usepackage{tikz-cd} 
\usepackage{enumitem} 
\usepackage{longtable}
\usepackage{mathrsfs} 
\usepackage{hyperref} 
\usepackage{hyperref}
\usepackage{booktabs}
\hypersetup{
    colorlinks=true,
    linkcolor=blue,
    filecolor=blue,      
    urlcolor=blue,
    linktocpage=true
}

\newtheorem{theorem}{Theorem}[section]

\theoremstyle{definition}
\newtheorem{definition}[theorem]{Definition}
\newtheorem{proposition}[theorem]{Proposition}

\newtheorem{remark}[theorem]{Remark}
\newtheorem{example}[theorem]{Example}
\newtheorem{examples}[theorem]{Examples}

\theoremstyle{remark}

\numberwithin{equation}{section}

\begin{document}

\title[Homomorphisms between iterated skew polynomial rings]{On homomorphisms and cv-polynomials between iterated Ore extensions}



\author{Mar\'ia Camila Ram\'irez}
\address{Universidad Nacional de Colombia - Sede Bogot\'a}
\curraddr{Campus Universitario}
\email{macramirezcu@unal.edu.co}
\thanks{}

\author{Armando Reyes}
\address{Universidad Nacional de Colombia - Sede Bogot\'a}
\curraddr{Campus Universitario}
\email{mareyesv@unal.edu.co}

\thanks{This work was supported by Faculty of Science, Universidad Nacional de Colombia - Sede Bogot\'a, Colombia [grant number 53880].}



\date{}

\dedicatory{Dedicated to the memory of Professor Nikolay A. Vavilov}

\begin{abstract}

Motivated by the study of homomorphisms and cv-polynomials presented by Rimmer \cite{Rimmer1978} in the case of Ore extensions of automorphism type, Ferrero and Kishimoto \cite{FerreroKishimoto1980} and Kikumasa \cite{Kikumasa1990} in the setting of these extensions of derivation type, and Lam and Leroy \cite{LamLeroy1992} in the context of Ore extensions of mixed type over division rings, in this paper we investigate these both notions for iterations of these extensions. We show that the iteration of some of the results presented in those papers is non-trivial, and illustrate our treatment with several noncommutative algebras.

\end{abstract}

\maketitle


\section{Introduction}

For $R$ a commutative ring, Gilmer \cite{Gilmer1968} determined all the automorphisms of the commutative polynomial ring $R[x]$ which fix $R$ elementwise. Parmenter \cite{Parmenter} extended this result to the case where $R$ is commutative and $\sigma$ is any automorphism of $R$. The case where $\sigma$ is the identity on any ring $R$ was considered by Coleman and Enochs \cite{ColemanEnochs1971} and Brewer and Rutter \cite{BrewerRutter1972}. These  three results are corollaries of Rimmer's paper \cite[Theorem 1]{Rimmer1978}, since for a unital ring (not necessarily commutative) $R$ and $\sigma$ any automorphism of $R$, he determined all the automorphisms of the Ore extension (introduced by Ore \cite{Ore1931, Ore1933}) of automorphism type $R[x;\sigma]$ that fix $R$ elementwise (he called them $R$-linear maps). As an application of his results, he investigated isomorphisms between different skew polynomial rings which preserve an underlying ring isomorphism \cite[Theorem 3]{Rimmer1978}. On the other hand, Ferrero and Kishimoto \cite{FerreroKishimoto1980} and Kikumasa \cite{Kikumasa1990} studied automorphisms of Ore extensions of derivation type $R[x;\delta]$. They discussed conditions on $f(x) \in R[x;\delta]$ for the $R$-linear map $R[x;\delta] \to R[x;\delta]$ defined by $x^k\mapsto f(x)^{k}$ to be a $R$-ring automorphism.

Related with this topic, Lam and Leroy \cite{LamLeroy1992} studied \textquotedblleft transformations\textquotedblright\ from one Ore extension of mixed type $D[x';\sigma',\delta']$ to another extension $D[x;\sigma,\delta]$ over a division ring $D$. They stated that a $D$-homomorphism $\phi$ from $D[x';\sigma', \delta']$ to $D[x;\sigma, \delta]$ is determined by $\phi(x') := p(x) \in D[x;\sigma, \delta]$. The choice of $p(x)$ must satisfy the condition  $p(x)a = \sigma'(a)p(x) + \delta'(a)$, for every $a \in D$. Since the polynomial $p(x)$ allows us to make a \textquotedblleft change of variables\textquotedblright\ (from $x'$ to $x$), they called $p(x)$ a {\em change-of-variable polynomial} (or {\em cv-polynomial} for short) respect to the quasi-derivation $(\sigma', \delta')$ on $D$. In their paper, they proved different results about these polynomials and their corresponding morphisms. For instance, they showed that if the cv-polynomial is non-constant, then it can be characterized with an appropriate quasi-derivation $(\sigma', \delta')$ on $D$ \cite[Theorem 2.12]{LamLeroy1992}. Another important result establishes when two Ore extensions are isomorphic under certain conditions \cite[Theorem 5.6]{LamLeroy1992}. Several researchers have investigated the topic of morphisms of Ore extensions in one indeterminate (e.g. \cite{ArmendarizKooPark1987, ChuangLeeLiuTsai2010, ChuangTsai2009, ChuangTsai2010, Chuang2013, Chun1993, Leroy1985, LeroyMatczuk1985, Leroy1995, Leroy2012, MartinezPenas2018, MartinezKschischang2019}, and references therein).

Motivated by the treatments described above, in this paper we explore the notions of homomorphism and cv-polynomial of iterated Ore extensions. We show that the iteration of some of their results is non-trivial, in fact, in some cases these results do not hold, and exemplify our results with several families of noncommutative algebras.

The paper is organized as follows. In Section \ref{HBIOE}, we start by recalling some preliminaries and examples of iterated Ore extensions that are necessary throughout the paper. Next, Section \ref{Homon=2} develop the theory of homomorphisms and cv-polynomials of two-step iterated Ore extensions. We introduce our proposal of cv-polynomial in Definition \ref{definitionhomooreiteretared} that extends the corresponding in Lam and Leroy's paper \cite{LamLeroy1992}. Example \ref{notgeneralization} shows that \cite[Theorem 2.12]{LamLeroy1992} is false unless we impose a additional conditions. Our extended version of this result is formulated in Theorem \ref{Theoremaversion2}. We also formulate Theorems \ref{theoreminvariann=2} and \ref{teo2.12_paran=2} that characterize quasi-derivations depending of the cv-polynomials. Then, Theorems \ref{isoOreextension} and \ref{injective_n=2} say how are the isomorphisms in the sense of these polynomials. We illustrate our results with Jordan plane $\mathcal{J}$ (Theorem \ref{autoJordan}) and its quantization (Theorem \ref{autoqJordan}). Section \ref{Homon=3} presents the study of homomorphisms and cv-polynomials of three-step iterated Ore extensions by considering the approach made in the previous section. As in Section \ref{Homon=2}, we find that the quasi-derivation $(\sigma'_i, \delta'_i)$ depends only on the choice of the cv-polynomial $p_i$ (Theorems \ref{theoreminvariann=3}, \ref{teo2.12_paran=3}, and \ref{isoOreextension=3}). Section \ref{Homocasen} contains general ideas for the study of morphisms and cv-polynomials in the case of $n$-step iterated Ore extensions. Finally, in Section \ref{future} we formulate some ideas for a possible future work concerning morphisms between noncommutative rings related with Ore extensions.

Throughout the paper, the symbol $R$ means an associative (not necessarily commutative) ring with identity. $R^{*}$ denotes the set of units of $R$, while the symbols $D$ and $\Bbbk$ denote a division ring and a field, respectively. $M_n(R)$ means the ring of matrices of size $n\times n$ with entries in the ring $R$. The term homomorphism will mean unital ring homomorphism.

\section{Preliminaries}\label{HBIOE}

We start by considering some preliminaries about iterated Ore extensions which are necessary throughout the paper.

Recall that if $\sigma: R\rightarrow R$ is an endomorphism of the ring $R$, an additive map $\delta:R\rightarrow R$ is called a \textit{$\sigma$-derivation of $R$} if	$\delta(r s)=\sigma(r) \delta(s) + \delta(r)s$, for all $r, s\in R$ (strictly speaking, this is the definition of {\em left} $\sigma$-{\em derivation}, but we will not need the concept of {\em right} $\sigma$-derivation, which is any additive map $\delta:R \to R$ satisfying the rule $\delta(rs) = \delta(r)\sigma(s) + r\delta(s)$). Notice that $\delta(1)=\delta(1\cdot 1)=\sigma(1)\delta(1)+\delta(1)1=2\delta(1)$, whence $\delta(1)=0$. The pair $(\sigma, \delta)$ is called a {\em quasi-derivation on} $R$ \cite[Definition 3.1]{BuesoTorrecillasVerschoren2003}.

Following Ore \cite{Ore1931, Ore1933} (see also \cite[p. 36]{GoodearlWarfield2004}), $S:=R[x;\sigma,\delta]$ is said to be an {\em Ore extension} of $R$ or a {\em skew polynomial ring of} $R$, and it consists of left polynomials $\sum_{i=0}^{n}r_ix^{i}$, $r_i\in R\ (x^0 := 1)$, which are added in the usual way and multiplied according to the rule $xr = \sigma(r)x + \delta(r)$, for all $r\in R$. For any non-zero element $f = \sum_{i=0}^{n}r_ix^{i} \in R[x; \sigma, \delta]$ with $r_n\neq 0$, the non-negative integer $n$ is called the {\it degree} of $f$, abbreviated ${\rm deg}(f)$, and the element $r_n$ is called the {\it leading coefficient} of $f$. The zero element of $R[x; \sigma, \delta]$ is defined to have degree $-\infty$ and leading coefficient $0$. If $\delta:= 0$, then the expression $R[x;\sigma]$ denotes the {\em Ore extension of endomorphism type}, and when $\sigma$ is the identity on $R$, $R[x; \delta]$ denotes the {\em Ore extension of derivation type}. Since its introduction, ring-theoretical, homological, and geometrical properties of Ore extensions have been investigated extensively in the literature (e.g., \cite{BrownGoodearl2002, BuesoTorrecillasVerschoren2003, Cohn1995, Fajardoetal2020, GoodearlLetzter1994, GoodearlWarfield2004, McConnellRobson2001, NinoRamirezReyes2020, NinoReyes2023, ReyesSarmiento2022, ReyesSuarez2020}, and references therein).

The construction of Ore extensions can be iterated. Indeed, if we consider $S_1 = R[x_1;\sigma_1, \delta_1]$ for a quasi-derivation $(\sigma_1,\delta_1)$ on $R$, and we assume that $(\sigma_2,\delta_2)$ is a quasi-derivation on $S_1$, then one may define the two-step iterated Ore extension
\[
S_2 = S_1[x_2;\sigma_2, \delta_2] = R[x_1;\sigma_1, \delta_1][x_2;\sigma_2, \delta_2],
\]
where the relations $x_i r = \sigma_i(r)x_i + \delta_i(r),\ i = 1, 2$, and $x_2 x_1 = \sigma_2(x_1)x_2 + \delta_2(x_1)$ hold. Similarly, if we assume that $(\sigma_3, \delta_3)$ is a quasi-derivation on $S_2$, then one can define the three-step iterated Ore extension
\[
S_3 = S_2[x_3;\sigma_3,\delta_3] = S_1[x_2;\sigma_2, \delta_2] [x_3;\sigma_3,\delta_3] = R[x_1;\sigma_1, \delta_1][x_2;\sigma_2, \delta_2] [x_3;\sigma_3,\delta_3],
\]
under relations $x_i r = \sigma_i(r)x_i + \delta_i(r),\ i = 1, 2, 3$, and $x_j x_i = \sigma_j(x_i)x_j + \delta_j(x_i),\ 1\le i < j\le 3$ hold. Repeating this procedure, one obtains the so-called {\em iterated Ore extensions}.

Next, we present some interesting noncommutative algebras that can be expressed as iterated Ore extensions which will be taken up later to illustrate our results.

\begin{example}\label{examplesOreiterated}
\begin{itemize}
\item [\rm (i)] The {\em Jordan plane} $\mathcal{J}(\Bbbk)$ introduced by Jordan \cite{Jordan2001} is the free $\Bbbk$-algebra defined by $\mathcal{J}:=\Bbbk\{x,y\}/\langle yx-xy-y^2\rangle$. An easy computation shows that $\mathcal{J}(\Bbbk)$ can be expressed as  $\Bbbk[y][x; \delta_2]$ with $\delta_2(y) = -y^2$ (c.f. \cite[Section 1]{Shirikov2005}). 

\item[\rm (ii)] Following \cite[Definition 7.41]{MansourSchork2015}, for an element $q\in \Bbbk^{*}$, the $q$-\textit{skew Jordan plane} $\mathcal{J}_q(\Bbbk)$ is defined as the free algebra $\Bbbk\{ x,y\} / \langle qyx - xy - y^2 \rangle$. This algebra can be written as the Ore extension $\Bbbk[y][x;\sigma_2, \delta_2]$ with $\sigma_2(y) = qy$ and $\delta_2(y) = -y^2$. 

\item [\rm (iii)] Let $p\in \Bbbk^{*}$. The \textit{single parameter quantum matrix algebra} $O_p(M_n(\Bbbk))$ has generating basis the set of indeterminates $\{x_{ij}\},\ 1\le i, j\le n$, subject to the relations given by
\[
x_{ij}x_{lm} = \begin{cases}
px_{lm}x_{ij}, & i > l, j = m,\\
px_{lm}x_{ij}, & i = l, j > m,\\
x_{lm}x_{ij}, & i > l, j < m,\\
x_{lm}x_{ij} + (p - p^{-1})x_{im}x_{lj}, & i > l, j > m.
\end{cases}
\] 

It can be seen that $
O_p(M_2(\Bbbk)) \cong \Bbbk[x_{11}][x_{12};\sigma_{12}][x_{21};\sigma_{21}][x_{22};\sigma_{22},\delta_{22}]$, with the maps given by 
\begin{align*}
    \sigma_{12}(x_{11}) &= px_{11}, \ \ \ \ \ \ \ \
     \sigma_{21}(x_{11}) = px_{11},\\
    \sigma_{22}(x_{21}) &= px_{21},\ \ \ \ \ \ \ \
    \sigma_{22}(x_{12}) = px_{12},\\
    \delta_{22}(x_{11}) &= (p- p^{-1})x_{21}x_{12}.
\end{align*}

Otherwise, the value of $\sigma_{ij}$ is the identity map, and for $1 \leq i,j \leq 2$, the maps $\delta_{ij}$ are zero \cite[Exercise 2V]{GoodearlWarfield2004}.

\end{itemize}
\end{example}

For a ring $R$, $\sigma$ an endomorphism of $R$ and an element $a\in R$, it is easy to see that the rule $\delta_{a,\sigma}(r) := ar - \sigma(r)a$, for every $r\in R$, defines a $\sigma$-derivation of $R$, and hence $(\sigma, \delta_{a, \sigma})$ is a quasi-derivation on $R$. $\delta_{a,\sigma}$ is called a $\sigma$-{\em inner derivation by} $a$. Any derivation on $R$ which is not an inner derivation is called an {\em outer derivation} \cite[p. 31]{GoodearlWarfield2004}.

A polynomial $f(x) \in R[x; \sigma, \delta]$ is said to be {\it right invariant} if $f(x)R[x; \sigma, \delta] \subseteq R[x; \sigma, \delta]f(x)$, and {\it right semi-invariant} if $f(x)R \subseteq Rf(x)$ \cite[p. 84]{LamLeroy1992}. Lemonnier's Theorem states that the Ore extension $\Bbbk[x;\sigma, \delta]$ has a non-constant semi-invariant polynomial if and only if the $\sigma$-derivation $\delta$ is quasi-algebraic, which means that $\delta$ satisfies one of the following equivalent conditions: there exist constants $b_i \in \Bbbk$, with $b_n = 1$, such that $\sum_{i=1}^{n}b_i\delta^i$ is a $\sigma^n$-inner derivation. (2) There exist constants $b_1, \ldots, b_n$, not all zero, such that $\sum_{i=1}^{n} b_i\delta^i$ is a $\sigma'$-inner derivation for some endomorphism $\sigma' \in \Bbbk$ \cite[Theorem 4.7]{LamLeroy1992}.

\section{Two-step iterated Ore extensions}\label{Homon=2}

We start with Definition \ref{definition1iterated} which is a direct extension of \cite[Definition 2.4]{LamLeroy1992}. As we will see, the ring of coefficients $R[x_1;\sigma_1,\delta_1]$ is the same for both extensions, and the cv-polynomial only acts on the second indeterminate. As expected, in this situation we get similar results to those obtained by Lam and Leroy \cite{LamLeroy1992}. 

\begin{definition}\label{definition1iterated}
Let $S' = R[x_1;\sigma_1,\delta_1][x'_2; \sigma'_2, \delta'_2]$ and $S = R[x_1;\sigma_1,\delta_1][x_2; \sigma_2, \delta_2]$ be two iterated Ore extensions. Consider the homomorphism $\phi: S' \to S$ given by $\phi(x_1) = x_1$ and $\phi(x'_2):= p(x_2) \in S$,  with $p(x_2)$ satisfying 
\[
\phi(x'_2x_1) = \sigma'_2(x_1)p(x_2) + \delta'_2(x_1)\quad {\rm and}\quad \phi(x'_2x_1) = p(x_2)x_1,
\]

i.e.,  
\begin{equation}\label{conditiontwoOre0}
   p(x_2)x_1 = \sigma'_2(x_1)p(x_2) + \delta'_2(x_1).
\end{equation}

Besides,
\[
x_1 r = \sigma_1(r)x_1 + \delta_1(r),\quad {\rm and}\quad p(x_2) r = \sigma'_2(r)p(x_2) + \delta'_2(r).
\]

We say that $p(x_2)$ is a {\em cv-polynomial for two-step iterated Ore extensions} respect to the quasi-derivation $(\sigma'_2, \delta'_2)$ on $R[x_1;\sigma_1,\delta_1]$.

Conversely, if a polynomial $p(x_2)\in S$  satisfies (\ref{conditiontwoOre0}) for a quasi-derivation $(\sigma'_2, \delta'_2)$ on $R[x_1;\sigma_1,\delta_1]$, then we define the homomorphism $\phi:S' \to S$ by $\phi(x'_2) = p(x_2)$.
\end{definition}

Let us see some easy examples that illustrate Definition \ref{definition1iterated}. Before, recall that for $u \in R^*$, $\sigma_u$ denotes the {\it inner automorphism of} $R$ {\em associated with} $u$ defined by $\sigma_u(r) = uru^{-1}$, for all $r \in R$. 

\begin{example}\label{firstexamples}
Let $S' =R[x_1;\sigma_1,\delta_1][x'_2; \sigma'_2, \delta'_2]$ and $S = R[x_1;\sigma_1,\delta_1][x_2; \sigma_2, \delta_2]$ as in Definition \ref{definition1iterated}.
\begin{enumerate}
\item [\rm (i)] If $p(x_2) := c \in R$, then $p(x_2)x_1 = cx_1$ and $p(x_2)r = cr = \sigma'_2(r)c + \delta'_2(r)$, whence if we take $\delta'_2 := \delta_{2(c,\sigma'_2)}(r) = cr - \sigma'_2(r)c$, it follows that $p(x_2)$ is a cv-polynomial. 

\item [\rm (ii)] Let $p(x_2) = ax_2 + b \in S$. Then
\[
p(x_2)x_1 = (ax_2 + b)x_1 = ax_2x_1 + bx_1 = a\sigma_2(x_1)x_2 + a\delta_2(x_1) + bx_1,
\]
    
and
\[
\sigma'_2(x_1)(ax_2 + b) + \delta'_2(x_1) = \sigma'_2(x_1)(ax_2) + \sigma'_2(x_1)b + \delta'_2(x_1).
\]

If $\sigma'_2(x_1) = \sigma'_{2(a, \sigma_2)}(x_1) := (\sigma_a \circ \sigma_2)(x_1) =a\sigma_2(x_1)a^{-1}$ and $\delta'_2 = a\delta_2 + \delta_{2(b,\sigma'_2)}$, an easy computation shows that $p(x_2)$ is a cv-polynomial. 

\item [\rm (iii)] Let $p(x_2) = x_2^2 + c \in S$. Suppose that $\sigma'_2 = \sigma_{2}^2$, $\delta'_2 = \delta_2^2 + \delta_{2(c,\sigma'_2)}$, and $ \sigma_2\delta_2 = -\delta_2\sigma_2 $. Since
\begin{align*}
    p(x_2)x_1 &= (x_2^2 + c)x_1 = x_2^2x_1 + cx_1\\
    &= \sigma_2^2(x_1)x_2^2 + \delta_2\sigma_2(x_1)x_2 + \sigma_2\delta_2(x_1)x_2 + \delta_2^2(x_1) + cx_1\\
    &= \sigma_2^2(x_1)x_2^2 + \delta_2^2(x_1) + cx_1, 
\end{align*}

and 
\begin{align*}
    \sigma'_2(x_1)p(x_2) + \delta'_2(x_1) &= \sigma'_2(x_1)(x^2_2 + c) + \delta'_2(x_1) = \sigma'_2(x_1)x^2_2 +  \sigma'_2(x_1)c + \delta'_2(x_1)\\
    &= \sigma^2_2(x_1)x^2_2 + \sigma^2_2(x_1)c + \delta_2^2(x_1) + cx_1 - \sigma^2_2(x_1)c \\
    &= \sigma^2_2(x_1)x_2^{2}+ \delta_2^2(x_1) + cx_1, 
\end{align*}

then $\phi(x'_2x_1) = p(x_2)x_1 = \sigma'_2(x_1)p(x_2) + \delta'_2(x_1)$, that is, $p(x_2)$ is a cv-polynomial respect to the quasi-derivation $\left(\sigma_2^2, \delta_2^2 + \delta_{2(c,\sigma'_2)}\right)$ on $R[x_1;\sigma_1,\delta_1]$.
\end{enumerate}
\end{example}

Now, we introduce the following definition of cv-polynomial more general than Definition \ref{definition1iterated}.

\begin{definition}\label{definitionhomooreiteretared}
Let $S' =R[x'_1;\sigma'_1,\delta'_1][x'_2; \sigma'_2, \delta'_2]$ and $S =R[x_1;\sigma_1,\delta_1][x_2; \sigma_2, \delta_2]$. If we consider the homomorphism $\phi: S' \to S$ given by $\phi(x'_1):= p_1(x_1) = p_1 \in R[x_1;\sigma_1, \delta_1]$ and $\phi(x'_2):= p_2(x_1, x_2) = p_2 \in S$, satisfying the relations 
\begin{align*}
    \phi(x'_2x'_1) = &\ \sigma'_2(\phi(x'_1))p_2 + \delta'_2(\phi(x'_1)) = p_2p_1,\\
    \phi(x'_1r) = &\ \sigma'_1(r)p_1 + \delta'_1(r)= p_1r,\quad {\rm for\ all}\ r\in R\\
    \phi(x'_2r) = &\ \sigma'_2(r)p_2 + \delta'_2(r) = p_2r, \quad {\rm for\ all}\ r\in R, 
\end{align*}

or equivalently,
\begin{align}
    p_2p_1 = &\ \sigma'_2(p_1)p_2 + \delta'_2(p_1),\label{conditiontwoOre1} \\
    p_1r = &\ \sigma'_1(r)p_1 + \delta'_1(r), \label{conditiontwoOre2}\\
    p_2r = &\ \sigma'_2(r)p_2 + \delta'_2(r), \label{conditiontwoOre3}
\end{align}

then we say that the polynomials $p_i$'s $(i = 1, 2)$ are the {\em cv-polynomials respect to the quasi-derivations} $(\sigma'_i, \delta'_i)$'s, $i = 1, 2$, on $R$ and $R[x_1';\sigma_1',\delta_1']$, respectively.

Conversely, if two polynomials $p_1\in R[x_1; \sigma_1, \delta_1]$ and $p_2 \in S$ satisfy conditions (\ref{conditiontwoOre1}) - (\ref{conditiontwoOre3}), where $(\sigma_1', \delta_1')$ and $(\sigma_2', \delta_2')$ are quasi-derivations on $R$ and $R[x_1';\sigma_1',\delta_1']$, respectively, then we define the homomorphism $\phi:S' \to S$ by $\phi(x'_i) = p_i$ ($i = 1, 2)$.
\end{definition}

The next example shows that constant polynomials are cv-polynomials (c.f. Example \ref{firstexamples} (i)). 

\begin{example}
Consider the Ore extensions $S'$ and $S$ as in Definition \ref{definitionhomooreiteretared}. Let $p_1 := c$ and $p_2 := d$ be elements of $R$. With the aim of showing that conditions (\ref{conditiontwoOre1}) - (\ref{conditiontwoOre3}) hold, take the inner derivation $\delta'_2(c) := \delta_{2(d,\sigma'_2)}(c) = dc - \sigma'_2(c)d$. Then 
\[
\phi(x'_2x'_1) = p_2p_1 = dc \quad {\rm and}\quad \sigma'_2(c)d + \delta'_2(c) = \sigma'_2(c)d + dc - \sigma'_2(c)d = dc.
\]

In addition, for any $r \in R$ and the inner derivation given by $\delta'_1 = \delta_{1(c,\sigma'_1)}$, it is clear that $\phi(x'_1r) = p_1r = cr$. Also, note that $\sigma'_1(r)c + \delta'_1(r) = \sigma'_1(r)c + cr - \sigma'_1(r)c = cr$. In the same way, $\phi(x'_2r) = p_2r = \sigma'_2(r)p_2 + \delta'_2(r) = dr$. These facts show that any pair of constant polynomials in $S$ correspond to a morphism in the sense of cv-polynomials. Finally, note that the values of the endomorphisms $\sigma'_i$'s are arbitrary, while the corresponding $\delta'_i$'s are determined uniquely by the values of $\sigma'_i$'s and the constant values associated with the polynomials $p_i$'s.
\end{example}

\begin{example}\label{Examplep2=c}
Throughout this example, we fix $p_2 = c \in R$ as a constant polynomial, and consider some possibilities for the polynomial $p_1$ in the indeterminate $x_1$. This case is similar to that one considered by Lam and Leroy \cite{LamLeroy1992}. Table \ref{firsttable}, contains the conditions to guarantee that we get homomorphisms and cv-polynomials in the sense of Definition \ref{definitionhomooreiteretared}.
\end{example}
{\tiny 
\begin{longtable}{|c|c|c|} 
\hline
\multicolumn{3}{ |c| }{ $\boldsymbol{p_2 = c}$ }\\
\hline 
\textbf{cv-polynomial} &  $(\sigma'_i, \delta'_i)$ & \textbf{Conditions}\\
\hline $\begin{array}{lr}  p_1 = d \ \end{array}$
&  $\delta'_1 = \delta_{1(d, \sigma'_1)}$, \quad 
$\delta'_2 = \delta_{2(c, \sigma'_2)}$&  \\
\hline $\begin{array}{lr}  p_1 = a_1x_1 + a_0 \end{array}$
& $\begin{cases}
\sigma'_1 = a_1\sigma_1a_1^{-1},\\
\delta'_1 = a_1\delta_1 + \delta_{1(a_0, \sigma'_1)}, 
\end{cases}$ \quad $\begin{cases}
\sigma'_2 = \sigma_{2c},\\
\delta'_2 =  \delta_{2(c, \sigma'_2)} 
\end{cases}$&  \\ 
\hline $\begin{array}{lr}  p_1 = a_2x_1^2 + a_0 \end{array}$
& $\begin{cases}
\sigma'_1 = a_2\sigma_1^2a_2^{-1},\\
\delta'_1 = a_2\delta_1^2 + \delta_{1(a_0, \sigma'_1)},\\
\end{cases}$ \quad $\begin{cases}
\sigma'_2 = \sigma_{2c},\\
\delta'_2 =  \delta_{2(c, \sigma'_2)} 
\end{cases}$& $\sigma_1\delta_1 = - \delta_1\sigma_1$\\
\hline $\begin{array}{lr} p_1 = a_nx_1^n + a_0 \end{array}$
& $\begin{cases}
\sigma'_1 = a_n\sigma_1^na_n^{-1},\\
\delta'_1 = a_n\delta_1^n + \delta_{1(a_{0}, \sigma'_1)},\\
\end{cases}$ \quad $\begin{cases}
\sigma'_2 = \sigma_{2c},\\
\delta'_2 =  \delta_{2(c, \sigma'_2)} 
\end{cases}$& $\sigma_1\delta_1 = - \delta_1\sigma_1$\\
\hline $\begin{array}{lr}  p_1 =f(x_1)x_1 + a_0, \\
\ \ f(x_1) = a_2x_1 + a_1 \end{array}$
& $\delta'_1 = a_2\delta_1^2 + a_1\delta_1 + \delta_{1(a_0, \sigma'_1)}$, \quad $\begin{cases}
\sigma'_2 = \sigma_{2c},\\
\delta'_2 =  \delta_{2(c, \sigma'_2)} 
\end{cases}$& $\begin{array}{lr} \ \ \ \ \ \ \sigma_1\delta_1 = - \delta_1\sigma_1, \\
f(x_1)\sigma_1(-) = \sigma'_1(-)f(x_1)\end{array}$\\ 
\hline
$\begin{array}{lr} p_1  = f(x_1)x_1 + b, \\
f(x_1) = \sum_{i=0}^{n}a_ix_1^i \end{array}$
& $\delta'_1 = \sum_{i=1}^{n}a_i\delta_1^i + \delta_{1(b, \sigma'_1)}$, \quad $\begin{cases}
\sigma'_2 = \sigma_{2c},\\
\delta'_2 =  \delta_{2(c, \sigma'_2)} 
\end{cases}$ & $\begin{array}{lr} \ \ \ \ \ \ \sigma_1\delta_1 = - \delta_1\sigma_1, \\
f(x_1)\sigma_1(-) = \sigma'_1(-)f(x_1)\end{array}$\\
\hline
$\begin{array}{lr}  p_1  =\sum_{i=0}^{n}a_ix_1^i \end{array}$
&  $\delta'_1 = \sum_{i=1}^{n}a_i\delta_1^i + \delta_{1(a_0, \sigma'_1)}$,  \quad $\begin{cases}
\sigma'_2 = \sigma_{2c}, \\
\delta'_2 =  \delta_{2(c, \sigma'_2)} 
\end{cases}$ 
&$\begin{array}{lr} \ \ \sigma_1\delta_1 = - \delta_1\sigma_1, \\
\ a_i\sigma_1^i(-) = \sigma'_1(-)a_i\end{array}$\\
\hline
\caption{cv-polynomials and quasi-derivations}
\label{firsttable}
\end{longtable}}

Recall that a $\sigma$-derivation $\delta$ of  a ring $R$ is said to be {\em algebraic} if there exists a non-zero polynomial $g(x) \in R[x;\sigma, \delta]$ such that $g(\delta) = 0$ (here, the evaluation of a polynomial $g(x)=\sum_{i=0}^{n}a_ix^i\in R$ at $\delta$ is defined to be the operator  $g(\delta) = \sum_{i=0}^{n} a_i\delta^i$ on $R[x;\sigma, \delta]$) \cite[p. 87]{LamLeroy1992}.

The following result is one of the most important formulated by Lam and Leroy.

\begin{proposition}[{\cite[Theorem 2.12]{LamLeroy1992}}]\label{LamLeroy1992Theorem2.12}
Let $p(x) = \sum_{i = 0}^{n} r_ix^{i} \in D[x;\sigma, \delta]$ be of degree $n\ge 0$, and let $(\sigma', \delta')$ be any quasi-derivation on $D$.
\begin{enumerate}
    \item [\rm (1)] If $p(x)$ is a cv-polynomial respect to $(\sigma', \delta')$, then $\delta' = (p - b_0)(\delta)$ + $\delta_{b_0, \sigma'}$, and if $n\ge 1$, then $\sigma' = I_{b_n} \circ \sigma^n$, where $I_{b_n}(r) := \sigma_{b_n}(r) = b_{n}rb_{n}^{-1}.$ 
    \item [\rm (2)] If $\delta$ is not an algebraic quasi-derivation, then $p(x)$ is a cv-polynomial respect to $(\sigma', \delta')$ if and only if $\delta' = (p - b_0)(\delta) + \delta_{b_0, \sigma'}$.
\end{enumerate}
\end{proposition}

The following examples show that Proposition \ref{LamLeroy1992Theorem2.12} is false unless we impose additional conditions (see Table \ref{firsttable}). Our version of this result is formulated below in Theorem \ref{Theoremaversion2}.

\begin{examples}\label{notgeneralization}
\begin{enumerate}
\item [\rm (i)] Consider the polynomials $p_1 = a_2x_1^2 + a_1x_1 + a_0\in D[x_1;\sigma_1,\delta_1]$ and $p_2 = c\in D$. Suppose that $p_1$ is a cv-polynomial respect to the quasi-derivation $(\sigma'_1, \delta'_1)$ on $D$. Then $p_1r = \sigma'_1(r)p_1 + \delta'_1(r)$ for any $r\in D$, whence
{\small{\begin{align}\label{equLamLeroy1}
        p_1 r &= (a_2x_1^2 + a_1x_1 + a_0)r = a_2x_1^2r + a_1x_1r + a_0r \notag \\
        &= a_2\sigma_1^2(r)x_1^2 + a_2\delta_1\sigma_1(r)x_1 + a_2\sigma_1\delta_1(r)x_1 + a_2\delta_1^2(r)+ a_1\sigma_1(r)x_1 + a_1\delta_1(r) + a_0r. \notag \\
        & = a_2\sigma_1^2(r)x_1^2 + [a_2\delta_1\sigma_1(r) + a_2\sigma_1\delta_1(r) + a_1\sigma_1(r)]x_1 + a_2\delta_1^2(r) + a_1\delta_1(r) + a_0r.
\end{align}
}}
    
On the other hand, 
\begin{align}\label{equLamLeroy2}
        \sigma'_1(r)p_1 + \delta'_1(r) &= \sigma'_1(r)a_2x_1^2 + \sigma'_1(r)a_1x_1 + \sigma'_1(r)a_0 + \delta'_1(r).
\end{align}

By Proposition \ref{LamLeroy1992Theorem2.12}, $\delta'_1 = (p_1-a_0)(\delta_1)+ \delta_{1(a_0,\sigma'_1)}$. Besides, as $p_1$ has degree 2, $\sigma'_1 = \sigma_{1a_2}\circ\sigma_1^2$. By replacing this last expression in (\ref{equLamLeroy2}), we get 
\begin{align}\label{equLamLeroy3}
        \sigma'_1(r)p_1 + \delta'_1(r) &= (\sigma_{1a_2}\circ\sigma_1^2)(r)a_2x_1^2 + (\sigma_{1a_2}\circ\sigma_1^2)(r)a_1x_1 + (\sigma_{1a_2}\circ\sigma_1^2)(r)a_0 \notag \\
        &\ \ + (p_1-a_0)(\delta_1)(r)+ \delta_{1(a_0,\sigma'_1)}(r) \notag \\
        &= a_2\sigma_1^2(r)a_2^{-1}a_2x_1^2+ a_2\sigma_1^2(r)a_2^{-1}a_1x_1 + a_2\sigma_1^2(r)a_2^{-1}a_0 \notag\\
        &\ \ + a_2\delta_1^2(r) + a_1\delta_1(r) + a_0r- a_2\sigma_1^2(r)a_2^{-1}a_0\notag \\
        &= a_2\sigma_1^2(r)x_1^2+ a_2\sigma_1^2(r)a_2^{-1}a_1x_1 + a_2\delta_1^2(r) + a_1\delta_1(r) + a_0r.
\end{align}
    
Nevertheless, it is clear that expressions (\ref{equLamLeroy1}) and (\ref{equLamLeroy3}) are different. Let us see explore other examples where the coefficient ring is not a division ring.

\item [\rm (ii)] Consider the polynomial $p = a_2x^2 + a_1x + a_0 \in \mathcal{J}(\Bbbk) \cong \Bbbk[y][x;\sigma, \delta]$ (Example \ref{examplesOreiterated}(i)). Suposse that $p$ is a cv-polynomial respect to the quasi-derivation $(\sigma', \delta')$ on $\Bbbk[y]$. Then $py = \sigma'(y)p + \delta'(y)$ for $y \in \Bbbk[y]$, whence
\begin{align}\label{equLamLeroy1.1}
        p y &= (a_2x^2 + a_1x + a_0)y = a_2x^2y + a_1xy + a_0y \notag \\
        &= a_2\sigma^2(y)x^2 + a_2\delta\sigma(y)x + a_2\sigma\delta(y)x + a_2\delta^2(y)+ a_1\sigma(y)x + a_1\delta(y) + a_0y. \notag \\
        & =  a_2yx^2 - a_2y^2x - a_2y^2x + a_2y^4+ a_1yx + a_1yx - a_1y^2 + a_0y.
\end{align}

Since 
\begin{align}\label{equLamLeroy2.1}
        \sigma'(y)p_1 + \delta'(y) &= \sigma'(y)a_2x^2 + \sigma'(y)a_1x + \sigma'(y)a_0 + \delta'(y).
\end{align}
    
By Proposition \ref{LamLeroy1992Theorem2.12}, $\delta' = (p-a_0)(\delta)+ \delta_{(a_0,\sigma')}$. Besides, as $p$ has degree 2, $\sigma' = \sigma_{a_2}\circ\sigma^2$. By replacing this last expression in (\ref{equLamLeroy2}), we get
\begin{align}\label{equLamLeroy3.1}
        \sigma'(y)p + \delta'(y) &= (\sigma_{a_2}\circ\sigma^2)(y)a_2x^2 + (\sigma_{a_2}\circ\sigma^2)(y)a_1x + (\sigma_{a_2}\circ\sigma^2)(y)a_0 \notag \\
        &\ \ + (p-a_0)(\delta)(y)+ \delta_{(a_0,\sigma')}(y) \notag \\
        &= a_2\sigma^2(y)a_2^{-1}a_2x^2+ a_2\sigma^2(y)a_2^{-1}a_1x + a_2\sigma^2(y)a_2^{-1}a_0 \notag\\
        &\ \ + a_2\delta^2(y) + a_1\delta(y) + a_0y- a_2\sigma^2(y)a_2^{-1}a_0\notag \\
        &= a_2\sigma^2(y)x^2+ a_2\sigma^2(y)a_2^{-1}a_1x + a_2\delta^2(y) + a_1\delta(y) + a_0y \notag \\
        &= a_2yx^2+ a_1yx + a_2y^4 - a_1y^2 + a_0y,
\end{align}
    
and it is clear that expressions (\ref{equLamLeroy1.1}) and (\ref{equLamLeroy3.1}) are different. 

\item [\rm (iii)] Let $p = a_2x^2 + a_1x+ a_0 \in \mathcal{J}_p(\Bbbk) \cong \Bbbk[y][x;\sigma, \delta]$ (Example \ref{examplesOreiterated}(ii)). Suppose that $p$ is a cv-polynomial respect to the quasi-derivation $(\sigma', \delta')$ on $\Bbbk[y]$. Then $py = \sigma'(y)p + \delta'(y)$ for $y \in \Bbbk[y]$, and
 \begin{align}\label{equLamLeroy1.1skew}
        p y &= (a_2x^2 + a_1x + a_0)y = a_2x^2y + a_1xy + a_0y \notag \\
        &= a_2\sigma^2(y)x^2 + a_2\delta\sigma(y)x + a_2\sigma\delta(y)x \notag \\
        &\ \ + a_2\delta^2(y)+ a_1\sigma(y)x + a_1\delta(y) + a_0y \notag\\
        &= a_2q^2yx^2 - a_2qy^2x - a_2qy^2x + a_2y^4 + a_1qyx - a_1y^2 + a_0y.
    \end{align}
    
By using that
    \begin{align}\label{equLamLeroy2.1skew}
        \sigma'(y)p + \delta'(y) &= \sigma'(y)a_2x^2 + \sigma'(y)a_1x + \sigma'(y)a_0 + \delta'(y),
    \end{align}
    
Proposition \ref{LamLeroy1992Theorem2.12} asserts that $\delta' = (p-a_0)(\delta)+ \delta_{(a_0,\sigma')}$. Besides, as $p$ has degree 2, $\sigma' = \sigma_{a_2}\circ\sigma^2$, and by replacing this last expression in (\ref{equLamLeroy2.1skew}), 
    \begin{align}\label{equLamLeroy3skew}
        \sigma'(y)p + \delta'(y) &= (\sigma_{a_2}\circ\sigma^2)(y)a_2x^2 + (\sigma_{a_2}\circ\sigma^2)(y)a_1x + (\sigma_{a_2}\circ\sigma^2)(y)a_0 \notag \\
        &\ \ + (p-a_0)(\delta)(y)+ \delta_{(a_0,\sigma')}(y) \notag \\
        &= a_2\sigma^2(y)a_2^{-1}a_2x^2+ a_2\sigma^2(y)a_2^{-1}a_1x + a_2\sigma^2(y)a_2^{-1}a_0 \notag\\
        &\ \ + a_2\delta^2(y) + a_1\delta(y) + a_0y- a_2\sigma^2(y)a_2^{-1}a_0\notag \\
        &= a_2\sigma^2(y)x_1^2+ a_2\sigma^2(y)a_2^{-1}a_1x + a_2\delta^2(y) + a_1\delta(y) + a_0y \notag \\
        &= a_2q^2yx^2 + a_1q^2yx + a_2y^4 - a_1y^2 + a_0y.
    \end{align}
    
Nevertheless, it is clear that expressions (\ref{equLamLeroy1.1skew}) and (\ref{equLamLeroy3skew}) are different. 
\end{enumerate}
\end{examples}

The following theorem is the corrected version of \cite[Theorem 2.12]{LamLeroy1992} that imposes two additional conditions. This is one of the most important results of the paper.

\begin{theorem}\label{Theoremaversion2}
   Let $p_1 = \sum_{i=0}^{n}a_ix_1^i \in S$ be of degree $n \geq 0$ and $(\sigma'_1, \delta'_1)$ any quasi-derivation on $R$. If $p_1$ is a cv-polynomial respect to $(\sigma'_1, \delta'_1)$, then $\delta'_1 = (p_1 - a_0)(\delta_1) + \delta_{(a_0, \sigma'_1)}$ and if $n \geq 1$, then $\sigma_1\delta_1 = - \delta_1\sigma_1$ and $a_i\sigma_1^i = \sigma'_1a_i$ for every $i > 0$. 
\end{theorem}
\begin{proof}
The first part of the assertion is precisely the content of Proposition \ref{LamLeroy1992Theorem2.12}. Now, if $n \geq 1$, since $p_1$ is a cv-polynomial, then for every $r \in R$, $p_1r = \sigma'_1(r)p_1 + \delta'_1(r)$, i.e., $p_1 r - \sigma'_1(r)p_1 = \delta'_1(r)$, and hence
{\small{
\begin{align}\label{equa212}
       \sum_{i=0}^{n}a_ix_1^i r - \sigma'_1(r)\sum_{i=0}^{n}a_ix_1^i &= [a_n\sigma_1^n(r) - \sigma'_1(r)a_n]x_1^n + [a_{n-1}\sigma_1^{n-1}(r) - \sigma'_1(r)a_{n-1}]x_1^{n-1} \notag \\
       &\ \ + \cdots + [a_1\sigma_1(r) - \sigma'_1(r)a_1]x_1 - \sigma'_1(r)a_0 + p_{\sigma_1,\delta_1} \notag \\
       &\ \ + \sum_{i=1}^{n}a_i\delta_1^i(r) + a_0r = \delta'_1(r),
\end{align}
}}
    
where $p_{\sigma_1,\delta_1}$ are the possible combinations between $\sigma_1$ and $\delta_1$. Since (\ref{equa212}) is equal to $\delta'_1(r) = (p_1 - a_0)(\delta_1)(r) + \delta_{(a_0, \sigma'_1)}(r) = \sum_{i=1}^{n}a_i\delta_1^i(r) + a_0r - \sigma'_1(r)a_0$, it follows that 
{\normalsize{
\begin{align*}
    &\left[a_n\sigma_1^n(r) - \sigma'_1(r)a_n\right]x_1^n + \left[a_{n-1}\sigma_1^{n-1}(r) - \sigma'_1(r)a_{n-1}\right]x_1^{n-1}\\
    &+ \cdots + \left[a_1\sigma_1(r) - \sigma'_1(r)a_1\right]x_1 + p_{\sigma_1,\delta_1} = 0,
\end{align*}
}}

i.e., $a_i\sigma^i_1(r) = \sigma'_1(r)a_i$ for $1\leq i \leq n$ and $p_{\sigma_1, \delta_1} = 0$. Therefore,  $\sigma_1\delta_1 = - \delta_1\sigma_1$.
\end{proof}

\begin{remark}
Rimmer \cite{Rimmer1978} also imposed the condition $a_i\sigma^i_1(r) = \sigma_1(r)a_i$ in his study of $R$-linear maps between Ore extensions.
\end{remark}

\begin{example}\label{Examplep1constant}
In Table \ref{secondtable}, we consider $p_1 = c\in R$ and $p_2$ a polynomial of degree one respect to the indeterminate $x_2$. As we see, if $p_2 = ax_1x_2 + b$ then $\sigma'_2$ is determined according to one of the following possibilities: either mixed compositions cancel each other or the condition $(a_1x_1)\sigma_2 = \sigma'_2(a_1x_1)$ holds.
\end{example}

\begin{example}\label{p1degreeoneandp2arbitrary}
In Table \ref{thirdtable}, $p_1$ is a degree one polynomial respect to $x_1$, and $p_2$ is a polynomial in two indeterminates $x_1$ and $x_2$. From Example \ref{Examplep1constant} we know that $\sigma'_i$ and $\delta'_i$ are determined only by the choice of $p_i$. 
\end{example}

The results obtained in the previous examples illustrate Theorem \ref{theoreminvariann=2}. Before, it is natural to say that a polynomial $p(x_1,x_2) \in S = R[x_1; \sigma_1, \delta_1][x_2;\sigma_2, \delta_2]$ is {\it right invariant} if $p(x_1,x_2)R[x_1; \sigma_1, \delta_1] \subseteq R[x_1; \sigma_1, \delta_1]p(x_1,x_2)$, and {\it right semi-invariant} if $p(x_1,x_2)R \subseteq Rp(x_1,x_2)$.

\begin{theorem}\label{theoreminvariann=2}
    Consider $x_2^n$ a monomial right invariant of degree  $n \geq 1$. Let $p_2(x_1,x_2)$ as $p_2(x_1, x_2) = f(x_1)x_2^n + g(x_1,x_2)$, where the degree of $g(x_1,x_2)$ is less than or equal to $n$. Then: 
    \begin{itemize}
\item [{\rm (1)}] $p_2$ is a cv-polynomial respect to the quasi-derivation $(\sigma'_2, \delta'_2)$ if and only if $g(x_1,x_2)$ is a cv-polynomial respect to $(\sigma'_2, \delta'_2)$ and $f(x_1)\sigma_2^n(-) = \sigma'_2(-)f(x_1)$.
\item[{\rm (2)}] If $p_2$ is a left multiple of $f(x_1)$, then $p_2$ is a cv-polynomial if and only if is invariant. 
    \end{itemize}
\end{theorem}
\begin{proof}
\begin{itemize}
\item [\rm (1)] If $p_2$ is a cv-polynomial respect to the quasi-derivation $(\sigma'_2, \delta'_2)$. Then
\begin{align}\label{invariantpart1}
        p_2p_1 & = \sigma'_2(p_1)p_2 + \delta'_2(p_1) = \sigma'_2(p_1)[f(x_1)x_2^n + g(x_1,x_2)] + \delta'_2(p_1) \\
            & = \sigma'_2(p_1)f(x_1)x_2^n + [\sigma'_2(p_1)g(x_1,x_2) + \delta'_2(p_1)] \notag \\
       & = \left[f(x_1)x_2^n + g(x_1,x_2)\right]p_1 \notag \\
        & = f(x_1)x_2^np_1 + g(x_1,x_2)p_1 = f(x_1)\sigma_2^n(p_1)x_2^n + g(x_1,x_2)p_1, \label{invariantpart2}
    \end{align}
    
since $x_2^n$ is a monomial invariant. Both expressions (\ref{invariantpart1}) and (\ref{invariantpart2}) show that $f(x_1)\sigma^n_2(-) = \sigma'_2(-)f(x_1)$ and 
\begin{equation}\label{equationg}
   g(x_1, x_2)p_1 = \sigma'_2(p_1)g(x_1,x_2) + \delta'_2(p_1). 
\end{equation}

Now, if $g(x_1,x_2)$ is a cv-polynomial respect to $(\sigma'_2, \delta'_2)$, then we obtain  $g(x_1, x_2)p_1 = \sigma'_2(p_1)g(x_1,x_2) + \delta'_2(p_1)$, and since $f(x_1)\sigma_2^n(-) = \sigma'_2(-)f(x_1)$, 
\begin{align*}
    p_2p_1 &= f(x_1)x_2^np_1 + g(x_1,x_2)p_1 = f(x_1)\sigma_2^n(p_1)x_2^n + g(x_1,x_2)p_1\\
    &= \sigma'_2(p_1)f(x_1)x_2^n + \sigma'_2(p_1)g(x_1,x_2) + \delta'_2(p_1)\\
    &= \sigma'_2(p_1)[f(x_1)x_2^n + g(x_1,x_2)] + \delta'_2(p_1) = \sigma'_2(p_1)p_2 + \delta'_2(p_1).
\end{align*}

\item[\rm (2)] If $p_2$ is a cv-polynomial respect to $(\sigma'_2, \delta'_2)$, then $p_2p_1 = \sigma'_2(p_1)p_2 + \delta'_2(p_1)$, and having in mind that $p_2$ is a left multiple of $f(x_1)$, $g(x_1,x_2) = 0$. In this way, $\delta'_2(p_1) = 0$ due to expression (\ref{equationg}). Besides, since $f(x_1)\sigma_2^n(-) = \sigma'_2(-)f(x_1)$, then $p_2p_1 = f(x_1)\sigma_2^n(p_1)x_2^n$, which shows that $p_2$ is invariant. Now, if $p_2$ is invariant and $g(x_1,x_2) = 0$, then $p_2p_1 = f(x_1)x_2^np_1 = f(x_1)\sigma_2^n(p_1)x_2^n$. If there exists $\sigma'_2$ with $f(x_1)\sigma^n_2(p_1) = \sigma'_2(p_1)f(x_1)$, then we get that $p_2$ is a cv-polynomial respect to the quasi-derivation $(\sigma'_2, \delta'_2= 0)$.
\end{itemize}
\end{proof}

\begin{landscape}
{\normalsize{
\begin{longtable}{|c|c|c|}\hline 
\multicolumn{3}{ |c| }{\bf $p_1 = c$ and $p_2$ with degree one in the indeterminate $x_2$}\\
\hline 
\textbf{cv-polynomials} &  $(\sigma'_i, \delta'_i)$ & \textbf{Conditions}\\
\hline $\begin{array}{lr}  
p_2 = ax_2 + b \end{array}$
& $ \delta'_1 = \delta_{1(c, \sigma'_1)}$ \quad $\begin{cases}
\sigma'_2 = a\sigma_2a^{-1}, \\
\delta'_2 = a\delta_2 + \delta_{2(b, \sigma'_2)} 
\end{cases}$ &   \\ 
\hline $\begin{array}{lr} 
p_2 = ax_1x_2 + b \end{array}$ 
& $\delta'_1 = \delta_{1(c, \sigma'_1)} $ \quad $\begin{cases}
\sigma'_2 = a\sigma_1\sigma_2a^{-1}, \\
\delta'_2 = ax_1\delta_2 + \delta_{2(b, \sigma'_2)} 
\end{cases}$& $\delta_1\sigma_2 = 0$\\
\hline $\begin{array}{lr} 
p_2 = ax_1x_2 + b\end{array}$
& $\delta'_1 = \delta_{1(c, \sigma'_1)} $ \quad $\delta'_2 = ax_1\delta_2 + \delta_{2(b, \sigma'_2)} $ & $(a_1x_1)\sigma_2(-) = \sigma'_2(-)(a_1x_1)$\\
\hline $\begin{array}{lr}
p_2 = f(x_1)x_2 + b, \\
f(x_1)= a_1x_1 + a_0 \end{array}$
& $ \delta'_1 = \delta_{1(c, \sigma'_1)} $ \quad $\begin{cases}
\delta'_2 = & a_1x_1\delta_2 + a_0\delta_2 + \delta_{2(b, \sigma'_2)}, \\
\ \ \ = & f(x_1)\delta_2 + \delta_{2(b, \sigma'_2)}
\end{cases}$& $f(x_1)\sigma_2(-) = \sigma'_2(-)f(x_1)$\\
\hline $\begin{array}{lr} 
p_2 = f(x_1)x_2 + b, \\
f(x_1)= a_2x_1^2 + a_1x_1 + a_0 \end{array}$
& $ \delta'_1 = \delta_{1(c, \sigma'_1)}$ \quad $ \delta'_2 = f(x_1)\delta_2  + \delta_{2(b, \sigma'_2)} $ & $\begin{array}{lr} f(x_1)\sigma_2(-) = \sigma'_2(-)f(x_1) \end{array}$\\
\hline $\begin{array}{lr}  p_2 = f(x_1)x_2 + g(x_1), \\ 
f(x_1) = \sum_{i=0}^{n}a_ix_1^i, \\
g(x_1) = \sum_{j=0}^{m}b_jx_1^j, \end{array}$
& $\delta'_1 = \delta_{1(c, \sigma'_1)},$ \quad $ \delta'_2 = f(x_1)\delta_2 + \delta_{2(g(x_1),\sigma'_2)} $ & $\begin{array}{lr} f(x_1)\sigma_2(-) = \sigma'_2(-)f(x_1) \end{array}$\\ 
\hline
\caption{cv-polynomials and quasi-derivations}
\label{secondtable}
\end{longtable}}}
\end{landscape}

\begin{landscape}
    {\small{
\begin{longtable}{|c|c|c|}
\hline
\multicolumn{3}{ |c| }{\bf $p_1$ with degree one in $x_1$ and $p_2$ arbitrary}\\
\hline 
\textbf{cv-polynomials} &  $(\sigma'_i, \delta'_i)$ & \textbf{Conditions}\\
\hline $\begin{array}{lr} p_1 = a_1x_1 + a_0, \\
p_2 = b_1x_2 + b_0\end{array}$
& $\begin{cases}
\sigma'_1 = a_1\sigma_1a_1^{-1},\\
\delta'_1 = a_1\delta_1 +\delta_{1(a_0, \sigma'_1)} 
\end{cases}$ \quad $\begin{cases}
\sigma'_2 = b_1\sigma_2b_1^{-1},\\
\delta'_2 = b_1\delta_2 + \delta_{2(b_0, \sigma'_2)} 
\end{cases}$&  \\ 
\hline $\begin{array}{lr}  p_1 = a_1x_1 + a_0, \\
p_2 = b_1x_1x_2 + b_0\\ \end{array}$
& $\begin{cases}
\sigma'_1 = a_1\sigma_1a_1^{-1},\\
\delta'_1 = a_1\delta_1  + \delta_{1(a_0, \sigma'_1)} 
\end{cases}$ \quad $\begin{cases}
\sigma'_2 = b_1\sigma_1\sigma_2b_1^{-1}\\
\delta'_2 = b_1x_1\delta_2 + \delta_{2(b_0, \sigma'_2)} 
\end{cases}$& $\delta_1\sigma_2 = 0$ \\ 
\hline $\begin{array}{lr} p_1 = a_1x_1 + a_0, \\
p_2 = f(x_1)x_2 + c, \\
f(x_1) = b_1x_1 + b_0\\
\end{array}$
& $\begin{cases}
\sigma'_1 = a_1\sigma_1a_1^{-1},\\
\delta'_1 = a_1\delta_1 + \delta_{1(a_0, \sigma'_1)} 
\end{cases}$ \quad $\begin{cases}
\delta'_2 = b_1x_1\delta_2  + b_0\delta_2 + \delta_{2(c, \sigma'_2)} \\
\ \ \ \ = f(x_1)\delta_2 + \delta_{2(c, \sigma'_2)}
\end{cases}$& $f(x_1)\sigma_2( - ) = \sigma'_2( - )f(x_1)$ \\ 
\hline $\begin{array}{lr} p_1 = a_1x_1 + a_0, \\ p_2 = f(x_1)x_2 + g(x_1),  \\ 
f(x_1) = \sum_{i=0}^{n}a_ix_1^i, \\
g(x_1) = \sum_{j=0}^{m}b_jx_1^j\\
\end{array}$
& $\begin{cases}
\sigma'_1 = a_1\sigma_1a_1^{-1},\\
\delta'_1 = a_1\delta_1 + \delta_{1(a_0, \sigma'_1)} 
\end{cases}$ \quad $ \delta'_2 = f(x_1)\delta_2 + \delta_{2(g(x_1),\sigma'_2)}$ & $\begin{array}{lr} f(x_1)\sigma_2( - ) = \sigma'_2( - )f(x_1) \end{array}$\\ 
\hline $\begin{array}{lr}  p_1 = a_1x_1 + a_0, \\ p_2 = b_2x_2^2 + b_0  \\ 
\end{array}$
& $\begin{cases}
\sigma'_1 = a_1\sigma_1a_1^{-1},\\
\delta'_1 = a_1\delta_1 + \delta_{1(a_0, \sigma'_1)} 
\end{cases}$ \quad $\begin{cases}
\sigma'_2 = b_2\sigma^{2}_2b_2^{-1},\\
\delta'_2 = b_2\delta_2^2 + \delta_{2(b_0,\sigma'_2)}\\
\end{cases}$& $\begin{array}{lr} \ \ \ 
\sigma_2\delta_2 = - \delta_2\sigma_2\\
b_2\sigma_2^2( - ) = \sigma'_2( - )b_2\end{array}$\\
\hline $\begin{array}{lr}  p_1 = a_1x_1 + a_0, \\ p_2 = b_1x_1x_2^2 + b_0  \\ 
\end{array}$
& $\begin{cases}
\sigma'_1 = a_1\sigma_1a_1^{-1}, \\
\delta'_1 = a_1\delta_1 +\delta_{1(a_0, \sigma'_1)} 
\end{cases}$ \quad $
\delta'_2 = b_1x_1\delta_2^2 + \delta_{2(b_0,\sigma'_2)} $& $\begin{array}{lr} \ \ \ \ \ \ \sigma_2\delta_2 = - \delta_2\sigma_2\\
b_1x_1\sigma_2^2(-) = \sigma^{'}_2(-)b_1x_1\end{array}$\\ 
\hline $\begin{array}{lr}  p_1 = a_1x_1 + a_0, \\ p_2 = f(x_1)x_2^n + g(x_1)  \\ 
f(x_1) = \sum_{i=0}^{n}a_ix_1^i\\
g(x_1) = \sum_{j=0}^{m}b_jx_1^j\\
\end{array}$
& $\begin{cases}
\sigma'_1 = a_1\sigma_1a_1^{-1}, \\
\delta'_1 = a_1\delta_1 + \delta_{1(a_0, \sigma'_1)} 
\end{cases}$ \quad $
\delta'_2 = f(x_1)\delta_2^n +  \delta_{2(g(x_1),\sigma'_2)}$ & $\begin{array}{lr} \ \ \ \ \ \ \sigma_2\delta_2 = - \delta_2\sigma_2\\
f(x_1)\sigma_2^n(-) = \sigma^{'}_2(-)f(x_1)\end{array}$\\ 
\hline $\begin{array}{lr} p_1 = a_1x_1 + a_0 , \\ p_2 = f(x_2)x_2 + g(x_1)  \\ 
f(x_2) = \sum_{i=0}^{n}a_ix_2^i\\
g(x_1) = \sum_{j=0}^{m}b_jx_1^j\\
\end{array}$
& $\begin{cases}
\sigma'_1 = a_1\sigma_1a_1^{-1}, \\
\delta'_1 = a_1\delta_1 + \delta_{1(a_0, \sigma'_1)} 
\end{cases}$ \quad $ \delta'_2 = \sum_{i=0}^{n} a_i\delta_2^i + \delta_{2(g(x_1),\sigma'_2)} $ & $\begin{array}{lr} \ \ \ \ \ \ \sigma_2\delta_2 = - \delta_2\sigma_2\\
f(x_2)\sigma_2(-) = \sigma^{'}_2(-)f(x_2)\end{array}$\\ 
\hline $\begin{array}{lr}  p_1 = a_1x_1 + a_0 , \\ p_2 =  \sum_{j=0}^{m}\sum_{i=0}^{n}b_{ij}x_1^ix_2^j\\
\end{array}$
& $\sigma'_1 = a_1\sigma_1a_1^{-1}, 
\delta'_1 = a_1\delta_1 + \delta_{1(a_0, \sigma'_1)}$,  \quad $\delta'_2 = \sum b_{ij}x_1^i\delta_2^i + \delta_{2(b_{i0}x_1^i,\sigma'_2)} $ & $\begin{array}{lr} \ \ \ \ \ \ \sigma_2\delta_2 = - \delta_2\sigma_2\\
b_{ij}x_1^i\sigma_2^j( - ) = \sigma^{'}_2( - )b_{ij}x_1^i\end{array}$\\ \hline
\caption{cv-polynomials and quasi-derivations}
\label{thirdtable}
\end{longtable}}}
\end{landscape}

Let $R$ be a ring and $(\sigma_2, \delta_2)$ a quasi-derivation on $R[x_1;\sigma_1,\delta_1]$. We say that $\delta_2$ is {\it algebraic} if, as expected, there exists a non-zero polynomial $g$ in the two-step iterated Ore extension $S = R[x_1;\sigma_1, \delta_1][x_2;\sigma_2,\delta_2]$ such that $g(x_1, \delta_2)=0$. The evaluation of a polynomial $g(x_1, x_2)=\sum_{i=0}^n q_i(x_1)x_2^i$ at $\delta_2$ is defined to be the operator $g(x_1, \delta_2) =\sum_{i=0}^n q_i(x_1)\delta_2^i$ on $S$.

The following result is the extension of Theorem \ref{Theoremaversion2}.

\begin{theorem} \label{teo2.12_paran=2}
Consider iterated Ore extensions $S'= R[x'_1;\sigma'_1,\delta'_1][x'_2; \sigma'_2, \delta'_2]$ and $S= R[x_1;\sigma_1,\delta_1][x_2; \sigma_2, \delta_2]$. Let $p_1(x_1) \in R[x_1;\sigma_1,\delta_1],  p_2(x_1, x_2) \in S$  given by $p_1(x_1) = \sum_{i = 0}^k a_{i}x_1^{i}$ and $p_2(x_1,x_2) = \sum_{i = 0}^n \sum_{j = 0}^m b_{ij}x_1^{i}x_2^{j}$. 
\begin{itemize}
\item [\rm (1)] If $p_2$ is a cv-polynomial respect to the quasi-derivation $(\sigma'_2, \delta'_2)$ then $\delta'_2 = (p_2 - B_{00})(\delta_2) + \delta_{2(B_{00}, \sigma'_2)}$, where $B_{00} = \sum_{i = 0}^nb_{i0}x_1^{i}x_2^{0}$, and if also $m \geq 1$, then $b_{ij}x_1^i\sigma_2^j(-) = \sigma'_2(-)b_{ij}x_1^i$ and $\sigma_2\delta_2 = -\delta_2\sigma_2$.
    
\item[\rm (2)] If $\delta_i$ is not an algebraic derivation for $1 \leq i \leq 2$, then $p_i$ is a cv-polynomial $1\leq i \leq 2$ respect to the quasi-derivation $(\sigma'_i, \delta'_i)$ if and only if $\delta'_1 = (p_1 - a_0)\delta_1 + \delta_{1(a_0, \sigma'_1)}$ and $\delta'_2 = (p_2 - B_{00})(\delta_2) + \delta_{2(B_{00}, \sigma'_2)}$, where $B_{00} = \sum_{i = 0}^nb_{i0}x_1^{i}x_2^{0}$.
\end{itemize}
\end{theorem}
\begin{proof}
\begin{itemize}
    \item [\rm (1)]Consider the homomorphism $\lambda: S \to \text{End}(R[x_1;\sigma_1, \delta_1], +)$, given by $\lambda(x_1) = x_1$, $\lambda(x_2) = \delta_2$ and for any $r \in R$, $\lambda(r)$  is defined as the left multiplication by $r$ on $R$. Suppose that  $p_2$ is a cv-polynomial respect to $(\sigma'_2, \delta'_2)$. Then
\begin{align}\label{equ1.2}
     p_2p_1 &=\phi(x'_2x'_1) = \sigma'_2(\phi(x'_1))p_2 + \delta'_2(\phi(x'_1))\\
     p_2r &= \phi(x'_2r) = \sigma'_2(r)p_2 + \delta'_2(r). \notag
\end{align}

If we evaluate the expression (\ref{equ1.2}) respect to the homomorphism $\lambda$, we get
\begin{equation}\label{equation1n=2}
    p_2(x_1,\delta_2)p_1(x_1)  = \sigma'_2(p_1(x_1))p_2(x_1, \delta_2) + \delta'_2(p_1(x_1)).
\end{equation}
Now, by evaluating at element 1 for $\delta_2$ in (\ref{equation1n=2}),
\[
p_2(x_1,\delta_2)p_1(x_1)  = \sigma'_2(p_1(x_1)) \sum_{i = 0}^n b_{i0}x_1^{i}x_2^{0}+ \delta'_2(p_1(x_1)), 
\]

since $\delta^j_2(1) = 0$ for $j \geq 1$. Therefore, by taking $ B_{00}= \sum_{i = 0}^nb_{i0}x_1^{i}x_2^{0}$, it follows that
\begin{align*}
   \delta'_2(p_1(x_1)) &=  p_2(x_1,\delta_2)p_1(x_1) - \sigma'_2(p_1(x_1)) B_{00}\\
   &= p_2(x_1,\delta_2)p_1(x_1) -B_{00}p_1(x_1) + B_{00}p_1(x_1)- \sigma'_2(p_1(x_1))B_{00}\\
   &= (p_2(x_1,\delta_2) - B_{00})p_1(x_1) + \delta_{2(B_{00}, \sigma'_2)}(p_1(x_1))\\
   &= (p_2(x_1,x_2) - B_{00})(\delta_2)(p_1(x_1)) + \delta_{2(B_{00}, \sigma'_2)}(p_1(x_1)).
\end{align*}

and thus $\delta'_2 = (p_2 - B_{00})(\delta_2) + \delta_{2(B_{00}, \sigma'_2)}$.

For the next part, let $m \geq 1$. Since $p_2$ is a cv-polynomial,  $p_2p_1 - \sigma'_2(p_1)p_2 = \delta'_2(p_1)$, and so
    \begin{align}\label{equa212.2}
       &\left(\sum_{i = 0}^n\sum_{j = 0}^m b_{ij}x_1^{i}x_2^{j}\right) \left(\sum_{i = 0}^k a_{i}x_1^{i}\right) - \sigma'_2\left(\sum_{i = 0}^k a_{i}x_1^{i}\right)\sum_{i = 0}^n\sum_{j = 0}^m b_{ij}x_1^{i}x_2^{j}\\
       &= [b_{nm}x_1^n\sigma_2^m(a_kx_1^k) - \sigma'_2(a_kx_1^k)b_{nm}x_1^n]x_2^m +\cdots -  \sigma'_2(p_1)B_{00}\notag\\
       &\ \ + p_{\sigma_2,\delta_2}+  \sum_{i = 0}^n\sum_{j = 0}^m b_{ij}x_1^{i}\delta_2^j(p_1) + B_{00}p_1 = \delta'_2(p_1), \notag
    \end{align}
    where $p_{\sigma_2,\delta_2}$ are the possible combinations between $\sigma_2$ and $\delta_2$. Since (\ref{equa212.2}) is equal to 
    \begin{align*}
        \delta'_2(p_1) &= (p_2 - B_{00})(\delta_2)(p_1) + \delta_{2(B_{00}, \sigma'_2)}(p_1)\\
        &=  \sum_{i = 0}^n\sum_{j = 0}^m b_{ij}x_1^{i}\delta_2^j(p_1) + B_{00}p_1 - \sigma'_1(p_1)B_{00},
    \end{align*}

necessarily the equality
\[
[b_{nm}x_1^n\sigma_2^m(a_kx_1^k) - \sigma'_2(a_kx_1^k)b_{nm}x_1^n]x_2^m +\cdots + p_{\sigma_2,\delta_2} = 0.
\]

must be satisfied. Therefore, $b_{ij}x_1^i\sigma_2^j(a_ix_1^i) = \sigma'_2(a_ix_1^i)b_{ij}x_1^i$ for $1\leq i \leq n$, $1\leq j \leq m$, $p_{\sigma_2, \delta_2} = 0$, and so $\sigma_2\delta_2 = - \delta_2\sigma_2$, as desired.

\item[(2)] The first implication has been proven in Theorem \ref{Theoremaversion2} and the previous part (1). Let us suppose that $\delta_2$ is not an algebraic derivation, and consider $q(x_1, x_2) = p_2(x_1, x_2) - B_{00}$. Then $\delta'_2 = q(x_1, \delta_2)+ \delta_{2(B_{00}, \sigma'_2)}$, and so $q(x_1,\delta_2) = \delta'_2 - \delta_{2(B_{00}, \sigma'_2)}$ is a $\sigma'_2$-derivation, which implies that for any $r \in R$,
\[
q(x_1, \delta_2)(p_1(x_1)r) = \sigma'_2(p_1(x_1))q(x_1, \delta_2)(r) + q(x_1,\delta_2)(p_1(x_1))r.
\]

In this way, 
\[
q(x_1, \delta_2)\lambda(p_1(x_1)) = \lambda(\sigma'_2(p_1(x_1)))q(x_1,\delta_2) + \lambda(q(x_1, \delta_2)(p_1(x_1))),
\]
which holds in the image of $\lambda$. Since $\delta_2$ is not algebraic, $\lambda$ is injective, and so there exist $\lambda^{-1}$, and
\[
q(x_1, x_2)(p_1(x_1)) = \sigma'_2(p_1(x_1))q(x_1,x_2) + q(x_1, \delta_2)(p_1(x_1)).
\]
By replacing $q(x_1, x_2)$ and $q(x_1, \delta_2)$, we get 
\begin{align*}
    (p_2(x_1, x_2) - B_{00})p_1(x_1) &= \sigma'_2(p_1(x_1))(p_2(x_1, x_2) - B_{00})\\
    &\ \ + (\delta'_2 - \delta_{2(B_{00}, \sigma'_2)})(p_1(x_1))\\
    &= \sigma'_2(p_1(x_1))p_2(x_1, x_2) - \sigma'_2(p_1(x_1))B_{00} \\
    &\ \ + \delta'_2(p_1(x_1))- \delta_{2(B_{00}, \sigma'_2)} (p_1(x_1))\\
    &= \sigma'_2(p_1(x_1))p_2(x_1, x_2) - \sigma'_2(p_1(x_1))B_{00}\\
    &\ \ + \delta'_2(p_1(x_1))- B_{00}p_1(x_1) + \sigma'_2(p_1(x_1))B_{00}\\
    &= \sigma'_2(p_1(x_1))p_2(x_1, x_2)  + \delta'_2(p_1(x_1)) - B_{00}p_1(x_1),
\end{align*}

and cancelling $B_{00}p_1(x_1)$, it follows that $p_2p_1 = \sigma'_2(p_1)p_2 + \delta'_2(p_1)$. 

Following the same ideas of part (2) of the proof above, but in this case with $q(x_1) = p_1(x_1) - a_0$, $\delta'_1 = p_1(\delta_1)+ \delta_{1(a_0,\sigma'_1)}$, and evaluating with the homomorphism  $\alpha: R[x_1;\sigma_1,\delta_1] \to \text{End}(R, +)$, $\alpha(x_1) = \delta_1$ and $\lambda(r)$ defined as the left multiplication by $r$ on $R$ for each $r\in R$, we can assert that $p_1r = \sigma'_1(p_1)r + \delta'_1(r)$. These facts show that $p_i$ ($i= 1,2$) is a cv-polynomial respect to the quasi-derivation $(\sigma'_i, \delta'_i)$ ($i= 1,2$).
\end{itemize}
\end{proof}

\begin{example}
Let $\mathcal{J'} = \Bbbk[y'][x'; \sigma'_2, \delta'_2]$ and $\mathcal{J} = \Bbbk[y][x; \sigma_2, \delta_2]$ (see Example \ref{examplesOreiterated} (i)). Consider the polynomials $p_1 = a_1y +a_0$ and $p_2 = b_1yx +b_0$. By Theorem \ref{teo2.12_paran=2}, we know that $\delta'_2=(p_2-B_{00})(\delta_2) + \delta_{2(B_{00},\sigma')}$ and $\sigma'_2 = a_1\sigma_2 a_1^{-1}$, with $\sigma_2(y)= y$ (recall that $\delta_2(y) = -y^2$, Example \ref{examplesOreiterated}). Since
\begin{align*}
    p_2p_1& = (b_1yx +b_0)(a_1y + a_0) = a_1b_1yxy+ a_0b_1yx + a_1b_0y + a_0b_0 \\
    &= a_1b_1y(\sigma_2(y)x + \delta_2(y)) + a_0b_1yx + a_1b_0y + a_0b_0\\
    &= a_1b_1y^2x - a_1b_1y^3+ a_0b_1yx + a_1b_0y + a_0b_0, 
\end{align*}

and
\begin{align*}
    \sigma'_2(p_1)p_2 + \delta'_2(p_1) &= \sigma'_2(a_1y+ a_0)(b_1yx + b_0) + \delta'_2(a_1y + a_0)\\
    &= a_1b_1y^2x + a_1b_0y + a_0b_1yx + a_0b_0 + \delta'_2(a_1y)+ \delta'_2(a_0)\\
    &= a_1b_1y^2x + a_1b_0y + a_0b_1yx + a_0b_0 + \sigma'_2(a_1)\delta'_2(y) + \delta'_2(a_1)y\\
    &= a_1b_1y^2x + a_1b_0y + a_0b_1yx + a_0b_0 +a_1[b_1y\delta_2(y)+b_0\sigma'_2(y) - yb_0]\\
    &= a_1b_1y^2x + a_1b_0y + a_0b_1yx + a_0b_0 - a_1b_1y^3,
\end{align*}

it follows that the polynomials $p_i$'s, $i = 1, 2$, are cv-polynomials respect to the quasi-derivation $(\sigma'_i, \delta'_i)$, $i =1, 2$. Note that the conditions for $\sigma'_1$ and $\delta'_1$ are satisfied.
\end{example}

Lam and Leroy \cite[p. 86]{LamLeroy1992} asserted that a $D$-homomorphism $\phi: D[x';\sigma',\delta'] \to D[x; \sigma, \delta]$ is {\it injective} if and only if the associated cv-polynomial $p(t) = \phi(t')$ has degree greater or equal $1$, and $\phi$ is {\it surjective} (respectively, {\it bijective}) if and only if $p(t)$ has degree one. Theorem \ref{isoOreextension} extends this result to two-step iterated Ore extensions.

\begin{theorem}\label{isoOreextension}
If $\phi$ is a homomorphism of two-step iterated Ore extensions as in Definition \ref{definitionhomooreiteretared}, then $\phi$ is an isomorphism if and only if the associated cv-polynomials $p_1$ and $p_2$ have degree one respect to the indeterminates $x_1$ and $x_2$, respectively.
\end{theorem}
\begin{proof}
Let $\phi$ be an isomorphism of two-step iterated Ore extensions. By Definition \ref{definitionhomooreiteretared}, $\phi(x'_1) = p_1 \in R[x_1;\sigma_1,\delta_1]$ and $\phi(x'_2):= p_2 \in S$. 

With the aim of showing that the cv-polynomials $p_1$ and $p_2$ necessarily have a degree greater or equal to one, consider the following cases:
\begin{itemize}
    \item [\rm (1)] $p_1 = c$ and $p_2 = c^{-1}dc$, with $c, d \in R$. It is clear that we can find elements $ f, g \in S'$ such that $f = dx'_1$ and $g = cx'_2$. Then $\phi(f) = dc = \phi(g)$, but since $f \neq g$, it contradicts the injectivity of $\phi$.
    \item [\rm (2)] $p_1 = c \in R$, and $p_2 \in S$ with degree greater than zero. Let $f, g \in S'$ such that $f = dx'_1$ and $g = dc$. It follows that $\phi(f) = dc = \phi(g)$, which contradicts again the injectivity of $\phi$.
    \item [\rm (3)] $p_1 \in R[x_1; \sigma_1, \delta_1]$ of degree greater than zero and $p_2 = d\in R$. Consider elements $f, g \in S'$ such that $f = d$ and $g = x'_2$. Again, $\phi(f) = d = \phi(g)$, which contradicts again the injectivity of $\phi$.
\end{itemize}

Assume that the cv-polynomials $p_1$ and $p_2$ have degree $n, k >1$, respectively. We write
\[
p_1 = \sum_{i=0}^n a_ix^{i}_1\ \ {\rm and}\ \ p_2 = \sum_{j=0}^k q_j(x_1)x^{j}_2,
\]
where $a_i \in R$, $q_j(x_1) \in R[x_1;\sigma_1,\delta_1]$ for $ 0 \leq i \leq n$ and $0 \leq j \leq k$, and $a_n, q_k(x_1)$ are non-zero elements. Since $\phi$ is surjective, for the indeterminate $x_1 \in S$ there exists a polynomial $g \in S'$ such that 
\[
x_1 = \phi(g) = \phi \left(\sum_{i,j= 0}^w c_{ij}x^{'i}_1x^{'j}_2\right) = \sum_{i,j= 0}^w c_{ij}\phi(x^{'i}_1x^{'j}_2) = \sum_{i,j= 0}^w c_{ij}p_1^{i}p_2^{j}.
\]

In this way, $w = 1$, the degree of $p_1$ is one, and the degree of $p_2$ is zero. However, since the cv-polynomials have degree greater than one, it follows that there is no polynomial $g \in S'$ with $\phi(g) = x_1$, which gives us a contradiction. Similarly, if we take the indeterminate $x_2 \in S$, we get that $w = 1$, $p_1$ must have degree zero or one, and $p_2$ must have degree one, which is again a contradiction. Therefore, both $p_1$ and $p_2$ must have degree exactly one.

Now, let $\phi:S' \to S$ be the homomorphism defined as $\phi(x'_1)= p_1 = ax_1 +b$ and $\phi(x'_2)= p_2 = f(x_1)x_2 + g(x_1)$, with $a \in R^*$, $b \in R$ and $f(x_1), g(x_1) \in R[x_1; \sigma_1, \delta_1]$. First of all, we show that both polynomials are cv-polynomials. With this aim, consider the following equalities:
\begin{align}\label{p2p1iso}
    \phi(x'_2x'_1) &= p_2p_1 = (f(x_1)x_2 + g(x_1))(ax_1 +b)\\
    &= f(x_1)x_2ax_1 + f(x_1)x_2b + g(x_1)ax_1 + g(x_1)b\notag \\
    &= f(x_1)[\sigma_2(ax_1)x_2 + \delta_2(ax_1)]+ f(x_1)x_2b + g(x_1)ax_1 + g(x_1)b \notag \\
    &= f(x_1)\sigma_2(ax_1)x_2 + f(x_1)\delta_2(ax_1)+ f(x_1)x_2b + g(x_1)ax_1 + g(x_1)b, \notag
\end{align}

and
\begin{align}\label{sigmap2piiso}
    \sigma'_2(\phi(x'_1))p_2 + \delta'_2(\phi(x'_1)) &= \sigma'_2(ax_1 +b)(f(x_1)x_2 + g(x_1)) + \delta'_2(ax_1 +b)\\
    &= \sigma'_2(ax_1)f(x_1)x_2 + \sigma'_2(ax_1)g(x_1) +\sigma'_2(b)f(x_1)x_2 \notag \\
     &\ \ + \sigma'_2(b)g(x_1) + \delta'_2(ax_1) + \delta'_2(b). \notag
\end{align}

Having in mind that (\ref{p2p1iso}) and (\ref{sigmap2piiso}) must coincide, let $\sigma'_2 := (\sigma_{2{f(x_1)}} \circ \sigma_2)$, that is,  $(\sigma_{2{f(x_1)}}\circ \sigma_2)( - ) = f(x_1)\sigma_2(-)(f(x_1))^{-1}$, so $f(x_1)$ is a unit, say $f(x_1) = \gamma \in R^{*}$. Since $\delta'_2 = (\gamma\delta_2 + \delta_{2{(g(x_1),\sigma'_2)}})$, it follows that $p_1 = ax + b$ and $p_2 = \gamma x_2 + g(x_1)$, whence
\begin{align} \label{p2p1iso2}
p_2p_1 &=  \gamma\sigma_2(ax_1)x_2 + \gamma\delta_2(ax_1)+ \gamma x_2b + g(x_1)ax_1 + g(x_1)b \notag \\
&= \gamma\sigma_2(ax_1)x_2 + \gamma\delta_2(ax_1)+ \gamma\sigma_2(b)x_2 + \gamma\delta_2(b)+ g(x_1)ax_1 + g(x_1)b, 
\end{align}
and
\begin{align}\label{sigmap2p1iso2}
    \sigma'_2(\phi(x'_1))p_2 + \delta'_2(\phi(x'_1)) &=   \sigma'_2(ax_1)\gamma x_2 + \sigma'_2(ax_1)g(x_1) +\sigma'_2(b)\gamma x_2 \notag\\
    &\ \ +  \sigma'_2(b)g(x_1) + \delta'_2(ax_1) + \delta'_2(b) \notag \\
    &=  \gamma \sigma_2(ax_1)\gamma^{-1} \gamma x_2 + \gamma \sigma_2(ax_1)\gamma^{-1}g(x_1) +\gamma \sigma_2(b)\gamma^{-1}\gamma x_2 \notag\\
    &\ \ +  \gamma \sigma_2(b)\gamma^{-1}g(x_1) + \gamma \delta_2(ax_1) + g(x_1)ax_1\notag \\
    &\ \  -  \gamma \sigma_2(ax_1)\gamma^{-1}g(x_1) + \gamma\delta_2(b) + g(x_1)b - \gamma\sigma_2(b)\gamma^{-1}g(x_1) \notag\\
    &=  \gamma \sigma_2(ax_1)x_2 +\gamma \sigma_2(b) x_2 +  \gamma \delta_2(ax_1) + g(x_1)ax_1 \\
    &\ \ +  \gamma\delta_2(b) + g(x_1)b. \notag
    \end{align}

If we compare (\ref{p2p1iso2}) and (\ref{sigmap2p1iso2}), then
$$
p_2p_1 = \sigma'_2(\phi(x'_1))p_2 + \delta'_2(\phi(x'_1)) =\sigma'_2(p_1)p_2 + \delta'_2(p_1),
$$
as desired. 

Now, since 
\begin{align*}
p_1r = &\ (ax + b)r = axr + br = a\sigma_1(r)x + a\delta_1(r) + br,\quad {\rm and} \\
\sigma'_1(r)p_1 + \delta'_1(r) = &\ \sigma'_1(r)ax + \sigma'_1(r)b + \delta'_1(r),
\end{align*}

then $\sigma'_1 = a \sigma_1 a^{-1}$ and $\delta'_1 = a\delta_1 + \delta_{1(b, \sigma'_1)}$. In this way,  
\begin{equation}\label{eq1}
    \sigma'_1(r)p_1 + \delta'_1(r) = a\sigma_1(r)a^{-1}ax + a\sigma_1(r)a^{-1}b + a\delta_1(r) + br - a \sigma_1(r) a^{-1}b,
\end{equation}

or equivalently, 
\[
p_1r =  a\sigma_1(r)x + a\delta_1(r) + br = \sigma'_1(r)p_1 + \delta'_1(r).
\]

All these facts guarantee that the polynomials $p_1 = ax_1 + b$ and $p_2 = \gamma x_2 +g(x_1)$ are cv-polynomials respect to the quasi-derivations $(\sigma'_1,\delta'_1) = (a\sigma_1a^{-1},a\delta_1 + \delta_{1(b,\sigma'_1)})$ and $(\sigma'_2, \delta'_2) = (\gamma \sigma_2 \gamma^{-1}, \gamma\delta_2 + \delta_{2{(g(x_1),\sigma'_2)}})$, respectively.

Finally, let us verify that $\phi$ is an isomorphism. We will show that there exists a homomorphism $\psi: S \to S'$ such that  $\psi \circ \phi = {\rm id}_{S'}$ and $\phi \circ \psi = {\rm id}_{S}$. Consider $x'_1 = \psi(\phi(x'_1))$, and the equalities
\[
x'_1 = \psi(\phi(x'_1)) = \psi(p_1) = \psi(ax_1 + b) = \psi(ax_1) + \psi(b).
\]

From the last expression, it follows that $\psi(x_1) = a^{-1}x'_1 - a^{-1}\psi(b)$.

In the same way, let us consider $x'_2 = \psi(\phi(x'_2))$. Since
\begin{align*}
x'_2 = &\ \psi(\phi(x'_2)) = \psi(p_2) = \psi(\gamma x_2 + g(x_1)) = \psi(\gamma x_2) + \psi(g(x_1))\\
= &\ \gamma \psi(x_2) + g(\psi(x_1)) = \gamma\psi(x_2) + g(a^{-1}x'_1-a^{-1}\psi(b)),
\end{align*}

we get $\psi(x_2) = \gamma^{-1}x'_2 - \gamma^{-1}g(a^{-1}x'_1-a^{-1}\psi(b))$. These facts show that $\psi \circ \phi = {\rm id}_{S'}$, where $p_1 = ax_1+ b$ and $p_2 = \gamma x_2 + g(x_1)$. In the same way, it can be shown that $\phi \circ \psi = {\rm id}_{S}$, whence the homomorphism defined by  $\phi(x'_1)= p_1 = ax_1 + b$ and $\phi(x'_2) = p_2= \gamma x_2 + g(x_1)$ is an isomorphism.
\end{proof}

Next, we characterize the automorphisms corresponding with cv-polynomials for the Jordan plane $\mathcal{J}$ (Theorem \ref{autoJordan}) and the $q$-skew Jordan plane $\mathcal{J}_p$ (Theorem \ref{autoqJordan}) .

\begin{theorem}\label{autoJordan}
If $\phi$ is an automorphism of the two-step iterated Ore extension $\mathcal{J}$, then the correspoding cv-polynomials are given by $\phi(y)= p_1= \gamma y$ and $\phi(x)= p_2=x+g(y)$ respect to $(\sigma_i, \delta_i)$, $i=1,2$ {\rm (}as in Example \ref{examplesOreiterated}{\rm (i))}, for some $\gamma \in \Bbbk^{*}$ and $g(y)\in \Bbbk[y]$. 
\end{theorem}
\begin{proof}
Consider $\phi$ an automorphism of $\mathcal{J}$. Then $\phi$ is an isomorphism of two-step iterated Ore extensions, so Theorem \ref{isoOreextension} establishes that its associated cv-polynomials have degree $1$. Let $\phi(y):=p_1 = \gamma y + \beta$ with $\gamma, \beta \in \Bbbk^*$ and $\phi(x):=p_2 = f(y) x + g(y)$ where $f(y), g(y) \in \Bbbk[y]$. The idea is to see that these polynomials satisfy conditions (\ref{conditiontwoOre1}) - (\ref{conditiontwoOre3}) respect to the quasi-derivations $(\sigma_i, \delta_i)$, $i = 1,2$, respectively.

Note that the equality $p_2p_1 = \sigma_2(p_1)p_2 + \delta_2(p_1)$ holds if the expressions
\begin{align*}
    p_2p_1 &= (f(y) x + g(y))(\gamma y + \beta) = f(y)x \gamma y + f(y) x \beta + g(y) \gamma y + g(y)\beta\\
    &= \gamma f(y)xy + \beta f(y)x + \gamma g(y)y + \beta g(y)\\
    &= \gamma f(y)(yx - y^2) + \beta f(y)x + \gamma g(y)y + \beta g(y)\\
    &= \gamma f(y)yx - \gamma f(y)y^2 + \beta f(y)x + \gamma g(y)y + \beta g(y), 
\end{align*}

and
\begin{align*}
  \sigma_2(p_1)p_2 + \delta_2(p_1) &= \sigma_2(\gamma y + \beta)(f(y) x + g(y)) + \delta_2(\gamma y + \beta)\\
  &= (\sigma_2(\gamma y) + \sigma_2(\beta))(f(y) x + g(y)) + \delta_2(\gamma y) + \delta_2(\beta)\\
  &= \gamma yf(y)x + \gamma y g(y) + \beta f(y)x + \beta g(y) \\ 
  &\ + \sigma_2(\gamma)\delta_2(y) + \delta_2(\gamma)y + \delta_2(\beta)\\
  &= \gamma f(y) yx + \gamma g(y)y + \beta f(y)x + \beta g(y) - \gamma y^2 
\end{align*}

coincide, i.e., $\gamma f(y) = \gamma$, whence $f(y) \in \Bbbk^*$ and necessarily $f(y) = 1$. Thus, $p_2 = x + g(y)$ and $p_1 = \gamma y + \beta$. It is straightforward to see that $ p_1r = \sigma_1(r)p_1 + \delta_1(r)$ and $p_2r = \sigma_1(r)p_2 + \delta_2(r)$ for any $r \in \Bbbk$. 

Now, since $\phi$ is an automorphism, there exists $\phi^{-1}$ such that $y = \phi^{-1}(\phi(y))$, that is, 
\[
y = \phi^{-1}(\phi(y)) = \phi^{-1}(p_1) = \phi^{-1}(\gamma y + \beta) = \phi^{-1}(\gamma y) + \phi^{-1}(\beta).
\]

From the last expression, it follows that $\phi^{-1}(\beta) = 0$, so $\beta = 0$. In addition, $\phi^{-1}(y) = \gamma^{-1} y$.\\

In the same way, let us consider $x = \phi^{-1}(\phi(x))$. Then
\begin{align*}
  x & = \phi^{-1}(\phi(x)) = \phi^{-1}(p_2) = \phi^{-1}(x + g(y)) = \phi^{-1}(x) + \phi^{-1}(g(y)) \\
  & = \phi^{-1}(x) + g(\phi^{-1}(y)) = \phi^{-1}(x) + g(\gamma^{-1}y), 
\end{align*}

so $\phi^{-1}(x) = x - g(\gamma^{-1}y)$. We conclude that the homomorphism defined by  $\phi(y)= p_1 = \gamma y$ and $\phi(x) = p_2= x + g(y)$, with $g(y) \in \Bbbk[y]$, is an automorphism of Jordan plane $\mathcal{J}$.
\end{proof}

\begin{example}\label{RemarkShirikov2005Proposition3.1}
Shirikov \cite[Proposition 3.1]{Shirikov2005} characterized the automorphisms of $\mathcal{J}$ (Example \ref{examplesOreiterated}(i)) defined over a field $\Bbbk$ of characteristic zero (he does not consider $\mathcal{J}$ as an iterated Ore extension). He proved that these automorphisms are given by $\phi(x) = \gamma x + g(y)$ and $\phi(y) = \gamma y$, for some $\gamma \in \Bbbk^*$ and $g(y) \in \Bbbk[y]$. If we compare with Theorem \ref{autoJordan}, there is a little difference concerning the coefficient of the indeterminate $x$. This is because our treatment concerns the definition of homomorphism through the notion of cv-polynomial. Next, we present an example of an automorphism of $\mathcal{J}$ that does not correspond to a pair of cv-polynomials.

Consider the automorphism $\phi$ of $\mathcal{J}$ given by the polynomials $\phi(x) = p_2 = 2x + y^2$ and $\phi(y) = p_1= 2y$. Since
\begin{align*}
p_2p_1 = &\ \sigma_2(p_1)p_2 + \delta_2(p_1) = (2x + y^2)(2y) = 4xy + 2y^3 \\
    = &\ 4(\sigma_2(y)x + \delta_2(y)) + 2y^3 = 4yx - 4y^2 + 2y^3,\\
    \sigma_2(2y)(2x + y^2) + \delta_2(2y)
    = &\ 2y(2x + y^2) -2y^2 = 4yx + 2y^3 - 2y^2.
\end{align*}

it follows that $  p_2p_1 \neq \sigma_2(2y)(2x + y^2) + \delta_2(2y)$, that is, the condition (\ref{conditiontwoOre1}) does not hold.
\end{example}

\begin{theorem}\label{autoqJordan}
If $\phi$ is an automorphism of the two-step iterated Ore extension $\mathcal{J}_q$, then the corresponding cv-polynomials are given by $\phi(y)= p_1 =\gamma y$ and $\phi(x)= p_2 = x$ respect to $(\sigma_i, \delta_i)$ {\rm (}$i=1,2${\rm )}, respectively {\rm (}as in Example \ref{examplesOreiterated}{\rm (ii))}, for some element $\gamma \in \Bbbk^{*}$.
\end{theorem}
\begin{proof}
    Let $\phi$ be an automorphism of $\mathcal{J}_q$. Then $\phi: \mathcal{J}_q \cong \Bbbk[y][x;\sigma_2, \delta_2] \to \mathcal{J}_q \cong \Bbbk[y][x;\sigma_2, \delta_2]$ acts as an isomorphism on the iterated Ore extension,  and by Theorem \ref{isoOreextension}, we have that the associated cv-polynomials have a degree of 1. We define $\phi(y):= p_1 = \gamma y + \beta$, where $\gamma$ and $\beta$ are non-zero elements in $\Bbbk$, and $\phi(x) := p_2 = f(y) x + g(y)$, such that $f(y)$ and $g(y)$ are elements of $\Bbbk[y]$. Let us see that the given polynomials satisfy the conditions  (\ref{conditiontwoOre1}) - (\ref{conditiontwoOre3}):
\begin{align*}
    p_2p_1 &= (f(y) x + g(y))(\gamma y + \beta) = f(y)x \gamma y + f(y) x \beta + g(y) \gamma y + g(y)\beta\\
    &= \gamma f(y)xy + \beta f(y)x + \gamma g(y)y + \beta g(y) \\
    &= \gamma f(y)(qyx - y^2) + \beta f(y)x + \gamma g(y)y + \beta g(y)\\
    &= \gamma qf(y)yx - \gamma f(y)y^2 + \beta f(y)x + \gamma g(y)y + \beta g(y).
\end{align*}

Now,
\begin{align*}
  \sigma_2(p_1)p_2 + \delta_2(p_1) &= \sigma_2(\gamma y + \beta)(f(y) x + g(y)) + \delta_2(\gamma y + \beta)\\
  &= (\sigma_2(\gamma y) + \sigma_2(\beta))(f(y) x + g(y)) + \delta_2(\gamma y) + \delta_2(\beta)\\
  &= \gamma qyf(y)x + \gamma qy g(y) + \beta f(y)x + \beta g(y) \\
  &\ \ + \sigma_2(\gamma)\delta_2(y) + \delta_2(\gamma)y + \delta_2(\beta)\\
  &= \gamma qf(y) yx + \gamma q g(y)y + \beta f(y)x + \beta g(y) - \gamma y^2.
\end{align*}

It is necessary to ensure that the equality $\gamma f(y) = \gamma$ holds, that is, $f(y)$ is a unit, and so $f(y) = 1$. In addition, we must satisfy $\gamma g(y)y = \gamma q g(y)y$. Given that $q \neq 1$, then $g(y) = 0$. In this way, $p_2 = x$ and $p_1 = \gamma y + \beta$. It is straightforward to see that for any $r \in \Bbbk$, $ p_1r = \sigma_1(r)p_1 + \delta_1(r)$ and $p_2r = \sigma_1(r)p_2 + \delta_2(r)$ hold. Hence, conditions (\ref{conditiontwoOre1}), (\ref{conditiontwoOre2}) and (\ref{conditiontwoOre3}) are satisfied.

Since $\phi$ is an automorphism, there exists $\phi^{-1}$ such that $y = \phi^{-1}(\phi(y))$. Then
\[
y = \phi^{-1}(\phi(y)) = \phi^{-1}(p_1) = \phi^{-1}(\gamma y + \beta) = \phi^{-1}(\gamma y) + \phi^{-1}(\beta).
\]

From the last expression, $\phi^{-1}(\beta) = 0$, and so $\beta = 0$. Also, we get $\phi^{-1}(y) = \gamma^{-1} y$.

Similarly, if we consider $x = \phi^{-1}(\phi(x))$, then
\[
x = \phi^{-1}(\phi(x)) = \phi^{-1}(p_2) = \phi^{-1}(x),
\]

and so $\phi^{-1}(x) = x$, i.e., we conclude that the homomorphism described by $\phi(y)= p_1 = \gamma y$ and $\phi(x) = p_2= x$ is an automorphism of the $q$-skew Jordan plane $\mathcal{J}_q$.
\end{proof}

Consider $S = R[x_1;\sigma_1, \delta_1][x_2;\sigma_2,\delta_2]$. The $\sigma_2$-derivation $\delta_2$ is {\it quasi-algebraic} if satisfies one of the following conditions:
\begin{itemize}
    \item[\rm (i)] There exist  $q_i(x_1) \in R[x_1;\sigma_1, \delta_1]$ with $q_n(x_1) = 1$, such that $\sum_{i=1}^{n}q_i(x_1)\delta^i_2$ is an $\sigma_2^n$-inner derivation.
    \item[\rm (ii)] There exist $q_i(x_1) \in R[x_1;\sigma_1, \delta_1]$, $1\leq i \leq n$ not all zero, such that $\sum_{i=1}^{n} q_i(x_1)\delta_2^i$ is an $\sigma'_2$-inner derivation for some endomorphism $\sigma'_2$. 
\end{itemize}   

Recall that if $\sigma$ is an endomorphism of $R$, the {\it inner order} of $\sigma$, denoted by $o(\sigma)$, is defined to be the smallest positive integer $n$ such that $\sigma^n$ is an inner authomorphism; if no such integer $n$ exists, we take $o(\sigma)$ to be $\infty$. In particular, if $\sigma$ is an endomorphism which is not an automorphism, we have by definition $o(\sigma) = \infty$ \cite[p. 85]{LamLeroy1992}. Note that for a $\delta_2$ derivation, algebraic implies quasi-algebraic. This is due to the fact that $g(x_1, \delta_2) = 0$, can be expressed as a $\sigma'_2$-inner derivation, for some appropriate $q \in R[x_1; \sigma_1, \delta_1]$.

\begin{theorem}\label{injective_n=2}
Consider iterated Ore extensions $S'= R[x'_1;\sigma'_1,\delta'_1][x'_2; \sigma'_2, \delta'_2]$ and $S= R[x_1;\sigma_1,\delta_1][x_2; \sigma_2, \delta_2]$. Assume that $\sigma_i\delta_i = -\delta_i\sigma_i$, and for $1 \leq i \leq 2$ either  
\begin{itemize}
    \item[\rm (1)] $p_2$ can be written as $p_2= q(x_2) + f(x_1)$ and $\sigma_i$ is an automorphism with $o(\sigma_i) = \infty$, or
    \item[\rm (2)] $\delta_i$ is not quasi-algebraic.    
\end{itemize}
If $\phi: S' \to S$ and $\psi: S \to S'$ are injective ring homomorphism, then both are isomomorphisms. 
\end{theorem}
\begin{proof}
Assume that one of the homomorphisms $\phi$ or $\psi$ is not an isomorphism. Then $\phi \circ \psi: S \to S$ is not an isomorphism, and by Theorem $\ref{isoOreextension}$, the associated cv-polynomials $p_1$ and $p_2$ have degree $\geq 2$. Suppose the condition $(1)$. First, we denote  $p_2 = \sum_{j = 0}^m b_{j}x_2^{j} + f(x_1)$ with  $m \geq 2$.  By Theorem \ref{teo2.12_paran=2}$(1)$, $b_{j}\sigma_2^j(-) = \sigma_2(-)b_{j}$, and since $\sigma_2$ is an automorphism, $\sigma_2^{j-1} = \sigma_{2(b_{j}^{-1})}$, for every $1 \leq j \leq m$, so $o(\sigma_2) \leq j-1 \leq m-1 <\infty$, a contradiction. In an analogous way, for $p_1 = \sum_{i=0}^{n} a_ix_1^i$ with $n\geq 2$,  $a_i\sigma_1^i = \sigma_1a_i$, and this implies that $\sigma_1^{i-1} = \sigma_{1(a_i^{-1})}$, for every $1 \leq i \leq n$, thus $o(\sigma_1) \leq i -1 \leq n-1 < \infty$.

Suppose that $\delta_2$ is not quasi-algebraic. Then $\delta_2$ is not algebraic and by Theorem \ref{teo2.12_paran=2}$(2)$, $\delta_2 = (p_2 - B_{00})(\delta_2) + \delta_{2(B_{00}, \sigma_2)}$, where $B_{00} = \sum_{i = 0}^nb_{i0}x_1^{i}x_2^{0}$. Since $p_2$ has degree $m \geq 2$, 
\[
\delta_2 - \sum_{i = 0}^n \sum_{j = 1}^m b_{ij}x_1^{i}\delta_2^{j} = \delta_2 -  \sum_{j = 1}^m q_j(x_1)\delta_2^{j}=\delta_{2(B_{00}, \sigma_2)},
\]

with $q_j(x_1) = b_{ij}x_1^i$ for $0 \leq i \leq n$, whence $\delta_2$ is a $\sigma'_2$-inner derivation with $\sigma'_2 = \sigma_2^m$, i.e, $\delta_2$ is quasi-algebraic, a contradiction. Following the same reasoning, Theorem \ref{teo2.12_paran=2}$(2)$ implies that $\delta_1 = (p_1 - a_0)\delta_1 + \delta_{1(a_0, \sigma'_1)}$, which shows that $\delta_1$ is quasi-algebraic, and we get a contradiction.
\end{proof}

\begin{remark}
It is worth asking whether the theorem is still valid if we change the hypothesis $(2)$ so that $\delta_2$ is not algebraic. This observation was studied by Lam and Leroy \cite[Remark 5.7]{LamLeroy1992}, where they showed in the case $n = 1$ that with a quasi-algebraic but not algebraic derivation two injective homomorphisms can be constructed, but these do not lead to an isomorphism. Now, in our case study, for $n=2$ it suffices to fix the ring $R[x_1]$ for both Ore extensions, which gives the same result.
\end{remark}

\begin{example}
For $\Bbbk$ a field of characteristic zero and an element $f(x)\in \Bbbk[x]$, Alev and Dumas \cite{AlevDumas1997} considered the algebra $\Lambda$ defined as the Ore extension 
    \[\Lambda:= \Lambda(f(x)) := \Bbbk[x]\left[y;\delta:= f(x) \frac{d}{dx}\right] = \Bbbk \langle x, y \mid yx - xy = f(x) \rangle.
    \]
Note that if $f(x) =  0$, then $\Lambda(0) = \Bbbk[x,y]$, and if $f(x) = 1$, then $\Lambda(1)$ is the first Weyl algebra. They investigated the group of automorphisms of $\Lambda$, and showed that for elements $f(x),\ g(x) \in \Bbbk[x]$, $\Lambda(f(x)) \cong \Lambda(g(x))$ if and only if $g(x) = \lambda f(\alpha x + \beta)$ for some elements $\lambda, \alpha \in \Bbbk^{*}$ and $\beta \in \Bbbk$ \cite[Proposition 3.6(i)]{AlevDumas1997}. As one can check after some computations, for some particular values of $\lambda, \alpha$ and $\beta$, the theory of cv-polynomials developed above illustrates Alev and Dumas's result.
\end{example}

To finish this section, we present some remarks about eigenvector interpretations for the coefficient vector of a cv-polynomial formulated by Lam and Leroy \cite[Theorem 2.16]{LamLeroy1992}. Before, recall that for the Ore extension $R[x; \sigma, \delta]$ and any element $r \in R$, $x^ir = \sum\limits_{j=0}^{i} f_{j}^{i}(r)x^i$, where $f_{j}^{i} \in {\rm End}(R, +)$ is the sum of all possible products with $j$ factors of $\sigma$ and $i-j$ factors of $\delta$. For example $f_j^j = \sigma^j$, $f_0^i=\delta^i$ and $f_{j-1}^j = \sigma^{j-1}\delta + \sigma^{j-2}\delta\sigma + \cdots + \delta\sigma^{j-1}$, in general $f_j^i = \binom{i}{j} \sigma^j\delta^{i-j}$. Let $M_n(r)$ denote the $n \times n$ matrix whose $(i,j)$-entry is $f_{j}^{i}(r)$ with $1 \leq i, j \leq n$, where $f_{j}^{i}(r)$ is taken to be zero if $i <j$, that is, $M_n(r)$ is a lower triangular matrix \cite[Section 2]{LamLeroy1988a}. 

\begin{proposition}[{\cite[Theorem 2.16]{LamLeroy1992}}]\label{LamLeroy1992Theorem2.16}
Let $p(x) = \sum\limits_{i=0}^{n} b_ix^i \in S = D[x; \sigma, \delta]$ be a polynomial of degree $n \geq 0$. Then the following assertions are equivalent:
\begin{itemize}
    \item [\rm (1)] $p(x)$ is a cv-polynomial.
    \item[\rm (2)] For any $a \in D$, and $j = 1, 2, \ldots, n$, we have $\sum\limits_{i=j}^{n} b_if_{j}^{i}(a) = b_n\sigma^n(a)b_{n}^{-1}b_j$.
    \item[\rm (3)] $[b_1, \ldots, b_n]$ is a left eigenvector for each matrix $M_n(a)$, where $a \in D$.
    \item[\rm (4)] $p(x)D \subseteq Dp(x) + D$.
\end{itemize}
If these conditions hold, then $up(x) + c$ is also a cv-polynomial, for any $u, c \in D$.
\end{proposition}

The following example illustrates Proposition \ref{LamLeroy1992Theorem2.16}.

\begin{example}
Consider the Example (2.10) presented in \cite[p. 87]{LamLeroy1992}. The polynomial $p(x) = x^2 + b_0 \in S = D[x; \sigma, \delta]$ is a cv-polynomial respect to the quasi-derivation $(\sigma^2, \delta^2 + \delta_{(c,\sigma^2)})$ on $D$, where $\sigma\delta = -\delta\sigma$. By Proposition \ref{LamLeroy1992Theorem2.16}(2), for any $a \in D$ and $j = 1, 2, \dotsc, n$, we get $\sum\limits_{i=j}^{n} b_if_{j}^{i}(a) = b_n\sigma^n(a)b_{n}^{-1}b_j$. In particular, for $j=1$,
\begin{align*}
        \sum\limits_{i=1}^{2} b_if_{1}^{i}(a) = b_1f_{1}^{1}(a) +  b_2f_{1}^{2}(a) = 0\sigma(a) + 1\sigma\delta(a) + 1\delta\sigma(a) = 0,
\end{align*}

and for $j=2,\ \sum\limits_{i=2}^{2} b_if_{2}^{i}(a) =  b_2f_{2}^{2}(a) = \sigma^2(a)$. Then, for $j = 1$, $\sum\limits_{i=1}^{2} b_if_{j}^{i}(a) = 0 = b_2\sigma^2(a)b_{2}^{-1}0$, and for $j = 2$,  $\sum\limits_{i=2}^{2} b_if_{2}^{i}(a) = \sigma^2(a) = b_2\sigma^2(a)b_{2}^{-1}b_2$ as Proposition \ref{LamLeroy1992Theorem2.16}(2) asserts. 

On the other hand, following Proposition \ref{LamLeroy1992Theorem2.16}(3), $[b_1 \ \ b_2]$ is a left eigenvector for each $a \in D$ of the matrix $M_2(a)$ with eigenvalue $b_2\sigma^2(a)b_2^{-1}$, whence
    \begin{align*}
        [b_1 \ \ b_2] \begin{bmatrix}f_1^1(a) & f_2^1(a) \\ f_1^2(a) & f_2^2(a)\end{bmatrix} &= b_2\sigma^2(a) b_2^{-1}[b_1 \ \ b_2]\\
        [0 \ \ 1] \begin{bmatrix} \sigma(a) & 0 \\ \sigma\delta(a) + \delta\sigma(a) & \sigma^2(a) \end{bmatrix} &= \sigma^2(a)[0 \ \ 1]\\
        [\sigma\delta(a) + \delta\sigma(a) \ \ \sigma^2(a)] &= [0 \ \ \sigma^2(a)] \\
        [0 \ \ \sigma^2(a)] &= [0 \ \ \sigma^2(a)].
    \end{align*}

Note that without the condition that $\sigma\delta(a) = -\delta\sigma (a)$, Proposition \ref{LamLeroy1992Theorem2.16} would be false. Examples \ref{exampleteo216} illustrates this situation.
\end{example}

\begin{examples}\label{exampleteo216}
\begin{enumerate}
\item [\rm (i)] Consider $p(x) = b_2x^2 + b_1x + b_0 \in S = D[x;\sigma, \delta]$ (Example (\ref{Examplep2=c})). This polynomial is a cv-polynomial respect to the quasi-derivation $(\sigma', \delta' = b_2\delta^2 + b_1\delta + \delta_{(b_0, \sigma')})$ whenever $\sigma\delta = - \delta\sigma$ and $b_i\sigma^i(a) = \sigma'(a)b_i$. Now, for any $a \in D$ and $j = 1, 2, \ldots, n$, by Proposition \ref{LamLeroy1992Theorem2.16}(2), we get $\sum\limits_{i=j}^{n} b_if_{j}^{i}(a) = b_n\sigma^n(a)b_{n}^{-1}b_j$. For $j=1$,
    \begin{align*}
        \sum\limits_{i=1}^{2} b_if_{1}^{i}(a) &=  b_1f_{1}^{1}(a) +  b_2f_{1}^{2}(a) = b_1\sigma(a) + b_2\sigma\delta(a) + b_2\delta\sigma(a) = b_1\sigma(a).
    \end{align*}

while for $j=2,\  \sum\limits_{i=2}^{2} b_if_{2}^{i}(a) =  b_2f_{2}^{2}(a) = b_2\sigma^2(a)$. In this way, for $j = 1$, we obtain $\sum\limits_{i=1}^{2} b_if_{1}^{i}(a) = b_1\sigma(a) \neq b_2\sigma^2(a)b_{2}^{-1}b_1$, and for $j = 2$, $\sum\limits_{i=2}^{2} b_if_{2}^{i}(a) = b_2\sigma^2(a) = b_2\sigma^2(a)b_{2}^{-1}b_2$.

According to Proposition \ref{LamLeroy1992Theorem2.16}(3), $[b_1 \ \ b_2]$ is a left eigenvector for each $a \in D$ of the matrix $M_2(a)$ with eigenvalue $b_2\sigma^2(a)b_2^{-1}$, which implies that
    \begin{align*}
        [b_1 \ \ b_2] \begin{bmatrix}f_1^1(a) & f_2^1(a) \\ f_1^2(a) & f_2^2(a)\end{bmatrix} &= b_2\sigma^2(a) b_2^{-1}[b_1 \ \ b_2]\\
        [b_1 \ \ b_2] \begin{bmatrix} \sigma(a) & 0 \\ \sigma\delta(a) + \delta\sigma(a) & \sigma^2(a) \end{bmatrix} &= b_2\sigma^2(a)b_2^{-1}[b_1 \ \ b_2]\\
        [b_1\sigma(a) + b_2\sigma\delta(a) + b_2\delta\sigma(a) \ \ b_2\sigma^2(a)] &= [b_2\sigma^2(a)b_2^{-1}b_1 \ \ b_2\sigma^2(a)] \\
        [b_1\sigma(a) \ \ b_2\sigma^2(a)] & = [b_2\sigma^2(a)b_2^{-1}b_1 \ \ b_2\sigma^2(a)].
    \end{align*}

However, it is clear that this equality is not necessarily true. 

\item [\rm (ii)] Consider the polynomial $p = b_2x^2 + b_1x + b_0 \in \mathcal{J}_p(\Bbbk) \cong \Bbbk[y][x;\sigma, \delta]$ (Example \ref{examplesOreiterated}(ii)). Suppose that $p$ is a cv-polynomial respect to the quasi-derivation respect to $(\sigma', \delta' = b_2\delta^2 + b_1\delta + \delta_{(b_0, \sigma')})$, whenever $\sigma\delta = - \delta\sigma$ and $b_i\sigma^i(y) = \sigma'(y)b_i$. For $y \in \Bbbk[y]$, and $j = 1, 2, \ldots, n$, by Proposition \ref{LamLeroy1992Theorem2.16}(2), $\sum\limits_{i=j}^{n} b_if_{j}^{i}(y) = b_n\sigma^n(y)b_{n}^{-1}b_j$, so for $j=1$,
\[
\sum\limits_{i=1}^{2} b_if_{1}^{i}(y) =  b_1f_{1}^{1}(y) +  b_2f_{1}^{2}(y) = b_1\sigma(y) + b_2\sigma\delta(y) + b_2\delta\sigma(y) = b_1\sigma(y) = b_1qy
\]

and for $j=2,\ \sum\limits_{i=2}^{2} b_if_{2}^{i}(y) =  b_2f_{2}^{2}(y) = b_2\sigma^2(y) = b_2q^2y$. In this way, for $j = 1,\ \sum\limits_{i=1}^{2} b_if_{1}^{i}(y) = b_1\sigma(y) = b_1qy \neq b_2\sigma^2(y)b_{2}^{-1}b_1 = b_1q^2y$, and for $j = 2$, $\sum\limits_{i=2}^{2} b_if_{2}^{i}(y) = b_2\sigma^2(y) = b_2q^2y =  b_2\sigma^2(y)b_{2}^{-1}b_2$.

Now, Proposition \ref{LamLeroy1992Theorem2.16}(3) asserts that $[b_1 \ \ b_2]$ is a left eigenvector for $y \in \Bbbk[y]$ of the matrix $M_2(y)$ with eigenvalue $b_2\sigma^2(y)b_2^{-1}$, that is, 
    \begin{align*}
        [b_1 \ \ b_2] \begin{bmatrix}f_1^1(y) & f_2^1(y) \\ f_1^2(y) & f_2^2(y)\end{bmatrix} &= b_2\sigma^2(y) b_2^{-1}[b_1 \ \ b_2]\\
        [b_1 \ \ b_2] \begin{bmatrix} \sigma(y) & 0 \\ \sigma\delta(y) + \delta\sigma(y) & \sigma^2(y) \end{bmatrix} &= b_2\sigma^2(y)b_2^{-1}[b_1 \ \ b_2]\\
        [b_1\sigma(y) + b_2\sigma\delta(y) + b_2\delta\sigma(y) \ \ b_2\sigma^2(y)] &= [b_2\sigma^2(y)b_2^{-1}b_1 \ \ b_2\sigma^2(y)] \\
        [b_1\sigma(y) \ \ b_2\sigma^2(y)] &\neq [b_2\sigma^2(y)b_2^{-1}b_1 \ \ b_2\sigma^2(y)] \\
         [b_1qy \ \ b_2q^2y] &\neq [b_1q^2y\ \ b_2q^2y].
    \end{align*} 
\end{enumerate}
\end{examples} 

Our corrected version of \cite[Theorem 2.16]{LamLeroy1992} is the content of Theorem \ref{LamLeroy1992Theorem2.16corrected}. 

\begin{theorem}\label{LamLeroy1992Theorem2.16corrected}
    Let $p_1(x) = \sum\limits_{i=0}^{n} a_ix_1^i \in S = R[x_1; \sigma_1, \delta_1]$ be a polynomial of degree $n \geq 0$. Then the following assertions are equivalent:
\begin{itemize}
    \item [\rm (1)] $p_1(x_1)$ is a cv-polynomial.
\item[\rm (2)] For any $r \in R$, and $j = 1, 2, \ldots, n$, we have $\sum\limits_{i=j}^{n} a_if_{j}^{i}(r) =  a_j\sigma_1^j(r)$ whenever $\sigma_1\delta_1 = -\delta_1\sigma_1$.
    \item[\rm (3)] $[a_1, \ldots, 0], \dotsc, [0, \ldots, a_i, \ldots, 0]$ are left eigenvectors for each matrix $M_n(r)$, where $r \in R$. 
    \item[\rm (4)] $p_1(x_1)R \subseteq Rp_1(x_1) + R$.
\end{itemize}
If these conditions hold, then $up_1(x_1) + c$ is also a cv-polynomial, for any $u, c \in R$.
\end{theorem}
\begin{proof}
If $p_1(x_1)$ has degree $n=0$, the conditions follow immediately. 
\begin{itemize}
\item $(1) \Rightarrow (2)$ Let $p_1(x)$ be a cv-polynomial respect to $(\sigma'_1, \delta'_1)$. For any $r \in R$, 
\begin{equation}\label{Teo2.16,1}
    p_1(x_1)r = \sum\limits_{i=0}^{n} a_ix_1^i r = \sum\limits_{i=0}^{n} a_i \sum\limits_{j=0}^{i}f_{j}^{i}(r) x_1^j = \sum\limits_{j=0}^{n} \left( \sum\limits_{i=j}^{n} a_if_{j}^{i}(r)\right)x_1^j.
\end{equation}

Now, by Theorem \ref{Theoremaversion2},
\begin{equation}\label{Teo2.16,2}
    p_1(x_1)r = \sigma'_1(r)p_1(x_1) + \delta'_1\notag = a_j\sigma_1^ja_j^{-1}\sum\limits_{i=0}^{n} a_ix_1^i + \delta'_1(r). 
\end{equation}

These two expressions show that identity is satisfied for each $j = 1, 2, \ldots, n$, whenever that $\sigma_1\delta_1 = - \delta_1\sigma_1$.

\item $(2) \Rightarrow (3)$ The matrix $M_n(r)$ that is composed by $(i,j)-$entry of the $f_j^i(r)$ with $1 \leq i,j \leq n$, so the equation in Part (2) can be expressed in the form
\begin{align*}
[a_1, \ldots, 0] M_n(r) &\ = a_1\sigma_1a_1^{-1} [a_1, \ldots, 0]\\
&\ \ \ \vdots \\
[0, \ldots, a_i, \ldots, 0] M_n(r) &\ = a_i\sigma_1^ia_i^{-1}[0, \ldots, a_i, \ldots, 0],
\end{align*}

for each $r \in R$ and $0\leq i \leq n$, and $[0, \ldots, a_i, \ldots, 0]$ is a left eigenvector for $M_n(r)$ with eigenvalue $a_i\sigma_1^ia_i^{-1}$.

\item $(3) \Rightarrow (4)$ For every $r \in R$, 
\[
[0, \ldots, a_i, \ldots, 0]M_n(r) = \alpha(r) [0, \ldots, a_i, \ldots, 0],
\]
for some eigenvalue $\alpha(r) \in R$. Thus, for $1 \leq j \leq n$, we obtain $\sum\limits_{i=j}^{n} a_if_{j}^{i}(r) =  (\alpha(r))a_j$. Using this fact in the expression (\ref{Teo2.16,1}),
\[
p_1(x_1)r \in \sum_{j=1}^n (\alpha(r))a_jx_1^j + R \subseteq R p_1(x_1) + R,\quad {\rm for\ any}\ r\in R.
\]

\item $(4) \Rightarrow (1)$ For any $r \in R$, we have determined constants $r_1, r_2 \in R$ such that $p_1(x_1)r = r_1p_1(x_1) + r_2$. By defining $\sigma'_1(r) = r_1$ and $\delta'_1(r) = r_2$, it is easy to see that $(\sigma_1', \delta_1')$ is a quasi-derivation on $R$, whence $p_1(x_1)$ is a cv-polynomial respect to $(\sigma'_1, \delta'_1)$.

\item To prove that the last statement of the theorem is true, consider Part (4). We have $p_1(x_1)R \subseteq Rp_1(x_1) + R$, and so
\[
(up_1(x_1) + c) R \subseteq u(Rp_1(x_1) + R) + cR
    \subseteq R(up_1(x_1) + c) + R,
\]

since this is also a cv-polynomial.
\end{itemize}
\end{proof}

\begin{example}
Consider again Example \ref{exampleteo216}. For any $r \in R$ and $j = 1, 2, \dotsc, n$, Theorem \ref{LamLeroy1992Theorem2.16corrected}(2) asserts that $\sum\limits_{i=j}^{n} a_if_{j}^{i}(r) = a_j\sigma^j(a)$, and hence
\begin{align*}
\sum\limits_{i=1}^{2} a_if_{1}^{i}(r) = &\ a_1f_{1}^{1}(r) +  a_2f_{1}^{2}(r) = a_1\sigma(r) + a_2\sigma\delta(r) + a_2\delta\sigma(r) = a_1\sigma(a),\quad j = 1, \\
\sum\limits_{i=2}^{2} a_if_{2}^{i}(r) =  &\ a_2f_{2}^{2}(r) = a_2\sigma^2(r),\quad j = 2,
\end{align*}

while
\begin{align*}
\sum\limits_{i=1}^{2} a_if_{1}^{i}(r) = a_1\sigma(r) = a_1\sigma(r)\quad {\rm and}\quad \sum\limits_{i=2}^{2} a_if_{2}^{i}(r) = a_2\sigma^2(r) = a_2\sigma^2(r).
\end{align*}

By Theorem \ref{LamLeroy1992Theorem2.16corrected}(3), $[a_1\ \ 0]$ and $[0 \ \ a_2]$ are left eigenvectors for each $r \in R$ of the matrix $M_2(r)$ with eigenvalues $a_1\sigma(r)a_1^{-1}$ and $a_2\sigma(r)^2a_2^{-1}$, respectively, so that
    \begin{align*}
        [a_1 \ \ 0] \begin{bmatrix}f_1^1(r) & f_2^1(r) \\ f_1^2(r) & f_2^2(r)\end{bmatrix} &= a_1\sigma(r)a_1^{-1}[a_1 \ \ 0]\\
        [a_1 \ \ 0] \begin{bmatrix} \sigma(r) & 0 \\ \sigma\delta(r) + \delta\sigma(r) & \sigma^2(r) \end{bmatrix} &= a_1\sigma(r)a_1^{-1}[a_1 \ \ 0]\\
        [a_1\sigma(r)  \ \ 0] &= [a_1\sigma(r)a_1^{-1}a_1 \ \ 0] \\
        [a_1\sigma(r)  \ \ 0] &= [a_1\sigma(r) \ \ 0].
    \end{align*}
    
In the same way,
    \begin{align*}
        [0 \ \ a_2] \begin{bmatrix}f_1^1(r) & f_2^1(r) \\ f_1^2(r) & f_2^2(r)\end{bmatrix} &= a_2\sigma(r)^2a_2^{-1}[0 \ \ a_2]\\
        [0 \ \ a_2] \begin{bmatrix} \sigma(r) & 0 \\ \sigma\delta(r) + \delta\sigma(r) & \sigma^2(r) \end{bmatrix} &= a_2\sigma(r)^2a_2^{-1}[0 \ \ a_2]\\
        [a_2\sigma\delta(r) + \delta\sigma(r) \ \ a_2\sigma^2(r)] &= [0 \ \ a_2\sigma(r)^2a_2^{-1}a_2] \\
        [0 \ \ a_2\sigma^2(r)] &= [0 \ \ a_2\sigma^2(r)],
    \end{align*}
    
which shows that Theorem \ref{LamLeroy1992Theorem2.16corrected}(4) holds.
\end{example}

\section{Three-step iterated Ore extensions}\label{Homon=3}

In this section, we consider the question on homomorphisms and cv-polynomials of three-step iterated Ore extensions. We follow the ideas presented in Section \ref{Homon=2}.

\begin{definition}\label{definitionhomooreiteretared3}
Let two iterated Ore extensions with three indeterminates given by $S' = R[x'_1;\sigma'_1, \delta'_1][x'_2;\sigma'_2, \delta'_2][x'_3;\sigma'_3, \delta'_3]$ and $S = R[x_1;\sigma_1, \delta_1][x_2;\sigma_2, \delta_2][x_3;\sigma_3, \delta_3]$. The homomorphism $\phi:S' \to S$ defined as $\phi(x'_1):= p_1(x_1) = p_1 \in R[x_1;\sigma_1, \delta_1]$, $\phi(x'_2):= p_2(x_1, x_2) = p_2 \in R[x_1;\sigma_1, \delta_1][x_2; \sigma_2, \delta_2]$ and $\phi(x'_3):= p_3(x_1, x_2, x_3) = p_3 \in S$, satisfies
\begin{align*}
\phi(x'_3x'_2) = &\ \sigma'_3(\phi(x'_2))p_3 + \delta'_3(\phi(x'_2)), \\
\phi(x'_3x'_1) = &\ \sigma'_3(\phi(x'_1))p_3 + \delta'_3(\phi(x'_1)), \\
\phi(x'_2x'_1) = &\ \sigma'_2(\phi(x'_1))p_2 + \delta'_2(\phi(x'_1)), \\
\phi(x'_1r) = &\ \sigma'_1(r)p_1 + \delta'_1(r), \\
\phi(x'_2r) = &\ \sigma'_2(r)p_2 + \delta'_2(r), \\ 
\phi(x'_3r) = &\ \sigma'_3(r)p_3 + \delta'_3(r), 
\end{align*}

for $r \in R$, or equivalently,
\begin{align*}
    p_3p_2 = &\ \sigma'_3(p_2)p_3 + \delta'_3(p_2), \\
    p_3p_1 = &\ \sigma'_3(p_1)p_3 + \delta'_3(p_1), \\
    p_2p_1 = &\ \sigma'_2(p_1)p_2 + \delta'_2(p_1), \\
    p_1r = &\ \sigma'_1(r)p_1 + \delta'_1(r), \\
    p_2r = &\ \sigma'_2(r)p_2 + \delta'_2(r), \\
    p_3r = &\ \sigma'_3(r)p_3 + \delta'_3(r).
\end{align*}

The polynomials $p_i$'s are the {\em cv-polynomials for iterated Ore extensions respect to the quasi-derivations} $(\sigma'_i, \delta'_i)$ for $i=1,2,3$, respectively.

Conversely, if a polynomial $p_i$ for $i =1,2,3$  satisfies the above conditions for quasi-derivations in the sense above, then we can define $\phi:S' \to S$ by $\phi(x'_i) = p_i$ with $i = 1,2,3$.
\end{definition}

\begin{example}
    Let $\phi:S' \to S$ be the homomorphism as in Definition \ref{definitionhomooreiteretared3}. We define 
    \begin{align*}
        \phi(x'_1) = &\ p_1(x_1) = a_1x_1 + a_0, \\
        \phi(x'_2) = &\ p_2(x_1, x_2) = b_{11}x_1x_2 + b_{00}, \\
        \phi(x'_3) = &\ p_3(x_1,x_2,x_3) = c_{111}x_1x_2x_3 + c_{000}.
    \end{align*}
    
Let us see that these are cv-polynomials respect to the quasi-derivations $(\sigma'_i, \delta'_i)$, $i = 1, 2, 3$. From Example \ref{p1degreeoneandp2arbitrary}, with $p_2 = f(x_1)x_2 + b_{00}$ where $f(x_1) = b_{11}x_1$, we know that $\sigma'_1 = a_1\sigma_1a_1^{-1}$, $\delta'_1 = a_1\delta_1 +\delta_{1(a_0, \sigma'_1)}$, $\delta'_2 = f(x_1)\delta_2  + \delta_{2(b_{00}, \sigma'_2)}$ and $f(x_1)\sigma_2(-) = \sigma'_2(-)f(x_1)$.\\
    
Let us show that $p_3p_1 = \sigma'_3(p_1)p_3 + \delta'_3(p_1)$. Consider $p_3  = q(x_1, x_2)x_3 +c_{000}$ with $q(x_1, x_2) = c_{111}x_1x_2$. Then
\begin{align}\label{p3p1example1n=3}
    p_3p_1&= (c_{111}x_1x_2x_3 + c_{000})(a_1x_1 + a_0)\notag \\
    &= c_{111}x_1x_2x_3a_1x_1 + c_{111}x_1x_2x_3a_0 + c_{000}a_1x_1 + c_{000}a_0\notag\\
    &= c_{111}x_1x_2\sigma_3(a_1x_1)x_3 + c_{111}x_1x_2\delta_3(a_1x_1) + c_{111}x_1x_2\sigma_3(a_0)x_3\notag\\
    &\ \ \ + c_{111}x_1x_2\delta_3(a_0) + c_{000}a_1x_1 + c_{000}a_0\notag\\
    &= q(x_1,x_2)\sigma_3(a_1x_1)x_3 + q(x_1,x_2)\delta_3(a_1x_1) + q(x_1,x_2)\sigma_3(a_0)x_3\notag \\
    &\ \ \ + q(x_1,x_2)\delta_3(a_0) + c_{000}a_1x_1 + c_{000}a_0.
\end{align}

Notice that
\begin{align}\label{sigma3p3p1example1n=3}
    \sigma'_3(p_1)p_3 + \delta'_3(p_1) &= \sigma'_3(a_1x_1 + a_0)(c_{111}x_1x_2x_3 + c_{000}) + \delta'_3(a_1x_1 + a_0) \notag \\
    &= \sigma'_3(a_1x_1)c_{111}x_1x_2x_3 + \sigma'_3(a_1x_1)c_{000} + \sigma'_3(a_0)c_{111}x_1x_2x_3 \notag\\
    &\ \ \ + \sigma'_3(a_0)c_{000} + \delta'_3(a_1x_1) + \delta'_3(a_0)\notag\\
    &= \sigma'_3(a_1x_1)q(x_1,x_2)x_3 + \sigma'_3(a_1x_1)c_{000} + \sigma'_3(a_0)q(x_1,x_2)x_3 \notag\\
    &\ \ \ + \sigma'_3(a_0)c_{000} + \delta'_3(a_1x_1) + \delta'_3(a_0).
\end{align}

With the aim of that conditions (\ref{p3p1example1n=3}) and (\ref{sigma3p3p1example1n=3}) coincide, let $\delta'_3 := q(x_1,x_2)\delta_3  + \delta_{3(c_{000}, \sigma'_3)}$ and assume that the condition $q(x_1,x_2)\sigma_3(-) = \sigma'_3(-)q(x_1,x_2)$ holds. For the equality $p_3p_2 = \sigma'_3(p_2)p_3 + \delta'_3(p_2)$, we get
\begin{align*}
    p_3p_2 &= (c_{111}x_1x_2x_3 + c_{000})(b_{11}x_1x_2 + b_{00})\\
    &= c_{111}x_1x_2\sigma_3(b_{11}x_1x_2)x_3 + c_{111}x_1x_2\delta_3(b_{11}x_1x_2)+ c_{111}x_1x_2\sigma_3(b_{00})x_3\\
    &\ \ \ + c_{111}x_1x_2\delta_3(b_{00})+ c_{000}b_{11}x_1x_2 + c_{000}b_{00}\\
     &= q(x_1,x_2)\sigma_3(b_{11}x_1x_2)x_3 + q(x_1,x_2)\delta_3(b_{11}x_1x_2)+ q(x_1,x_2)\sigma_3(b_{00})x_3\\
    &\ \ \ + q(x_1,x_2)\delta_3(b_{00})+ c_{000}b_{11}x_1x_2 + c_{000}b_{00}, 
\end{align*}

and
\begin{align*}
    \sigma'_3(p_2)p_3 + \delta'_3(p_2) &= \sigma'_3(b_{11}x_1x_2)c_{111}x_1x_2x_3 + \sigma'_3(b_{11}x_1x_2)c_{000}\\
    &\ \ \ + \sigma'_3(b_{00})c_{111}x_1x_2x_3 +\sigma'_3(b_{00})c_{000} + \delta'_3(b_{11}x_1x_2) + \delta'_3(b_{00})\\
    &= \sigma'_3(b_{11}x_1x_2)q(x_1,x_2)x_3 + \sigma'_3(b_{11}x_1x_2)c_{000}\\
    &\ \ \ + \sigma'_3(b_{00})q(x_1,x_2)x_3 +\sigma'_3(b_{00})c_{000} + \delta'_3(b_{11}x_1x_2) + \delta'_3(b_{00}).
\end{align*}

It is straightforward to see that by using $\delta'_3$ and the condition assumed above, we get the desired equality.
\end{example}

As in the case of three-step iterated Ore extensions, a polynomial $p(x_1,x_2, x_3) \in S = R[x_1; \sigma_1, \delta_1][x_2;\sigma_2, \delta_2][x_3;\sigma_3, \delta_3]$ is said to be {\it right invariant} if $$p(x_1,x_2,x_3)R[x_1; \sigma_1, \delta_1][x_2; \sigma_2, \delta_2] \subseteq R[x_1; \sigma_1, \delta_1][x_2; \sigma_2, \delta_2]p(x_1,x_2,x_3),$$ and {\it right semi-invariant} if $p(x_1,x_2,x_3)R \subseteq R p(x_1,x_2,x_3)$.

The following result is the natural extension of Theorem \ref{theoreminvariann=2}.
\begin{theorem}\label{theoreminvariann=3}
    Let $x_3^n$ be a monomial invariant of degree $n \geq 1$. Let $p_3(x_1,x_2,x_3)$ with $p_3(x_1, x_2,x_3) = f(x_1,x_2)x_3^n + g(x_1,x_2,x_3)$ where the degree of $g(x_1,x_2,x_3)$ is less than or equal to $n$. Then:
\begin{itemize}
\item [{\rm (1)}] $p_3$ is a cv-polynomial respect to $(\sigma'_3, \delta'_3)$ if and only if $g(x_1,x_2,x_3)$ is a cv-polynomial respect to $(\sigma'_3, \delta'_3)$ and $f(x_1,x_2)\sigma_3^n(-) = \sigma'_3(-)f(x_1,x_2)$.
\item[{\rm (2)}] If $p_3$ is a left multiple of $f(x_1,x_2)$, then $p_3$ is a cv-polynomial if and only if is invariant. 
    \end{itemize}
\end{theorem}
\begin{proof}
The ideas are completely analogous to the proof of Theorem \ref{theoreminvariann=2}. 
\begin{enumerate}
\item [\rm (1)] First, assume that $p_3$ is a cv-polynomial respect to the quasi-derivation $(\sigma'_3, \delta'_3)$. Then $p_3p_2 = \sigma'_3(p_2)p_3 + \delta'_3(p_2)$, and
\begin{align} \label{invariantpart1n=3}
            p_3p_2
            &= \sigma'_3(p_2)[f(x_1,x_2)x_3^n + g(x_1,x_2,x_3)] + \delta'_3(p_2) \notag\\
            &= \sigma'_3(p_2)f(x_1,x_2)x_3^n + [\sigma'_3(p_2)g(x_1,x_2,x_3) + \delta'_3(p_2)]
\end{align}

Now,
\begin{align}\label{invariantpart2n=3}
        p_3p_2 &= [f(x_1,x_2)x_3^n + g(x_1,x_2,x_3)]p_2 = f(x_1,x_2)x_3^np_2 + g(x_1,x_2,x_3)p_2 \notag\\
        &= f(x_1,x_2)\sigma_3^n(p_2)x_3^n + g(x_1,x_2,x_3)p_2,
\end{align}
    
since $x_3^n$ is a monomial invariant. From (\ref{invariantpart1n=3}) and (\ref{invariantpart2n=3}) it follows that $f(x_1,x_2)\sigma^n_3(-) = \sigma'_3(-)f(x_1,x_2)$, and
\begin{equation}\label{equationgn=3}
   g(x_1, x_2,x_3)p_2 = \sigma'_3(p_2)g(x_1,x_2,x_3) + \delta'_3(p_2). 
\end{equation}

Conversely, if $g(x_1,x_2,x_3)$ is a cv-polynomial respect to $(\sigma'_3, \delta'_3)$, then $g(x_1, x_2,x_3)p_2 = \sigma'_3(p_2)g(x_1,x_2,x_3) + \delta'_3(p_2)$, and having in mind that $f(x_1,x_2)\sigma_3^n(-) = \sigma'_3(-)f(x_1,x_2)$, 
\begin{align*}
    p_3p_2 &= f(x_1,x_2)x_3^np_2 + g(x_1,x_2,x_3)p_2\\
    &= f(x_1,x_2)\sigma_3^n(p_2)x_3^n + g(x_1,x_2,x_3)p_2\\
    &= \sigma'_3(p_2)f(x_1,x_2)x_3^n + \sigma'_3(p_2)g(x_1,x_2,x_3) + \delta'_3(p_2)\\
    &= \sigma'_3(p_2)[f(x_1,x_2)x_3^n + g(x_1,x_2,x_3)] + \delta'_3(p_2) = \sigma'_3(p_2)p_3 + \delta'_3(p_2).
\end{align*}

\item[\rm (2)] If $p_3$ is a cv-polynomial respect to the quasi-derivation $(\sigma'_3, \delta'_3)$, then $p_3p_2 = \sigma'_3(p_2)p_3 + \delta'_3(p_2)$. Since $p_3$ is a left multiple of $f(x_1,x_2)$, we get $g(x_1,x_2,x_3) = 0$. Due to the equation (\ref{equationgn=3}), $\delta'_3(p_2) = 0$, and besides $f(x_1,x_2)\sigma_3^n(-) = \sigma'_3(-)f(x_1,x_2)$, whence
\[
p_3p_2 = \sigma'_3(p_2)f(x_1,x_2)x_3^n = f(x_1,x_2)\sigma_3^n(p_2)x_3^n,
\]
which shows that $p_3$ is invariant. 

Now, assume that $p_3$ is invariant. By assumption, $g(x_1,x_2,x_3) = 0$, and so $p_3p_2 = f(x_1,x_2)x_3^np_2 = f(x_1,x_2)\sigma_3^n(p_2)x_3^n$. If there exist $\sigma'_3$ such that $f(x_1,x_2)\sigma^n_3(p_2) = \sigma'_3(p_2)f(x_1,x_2)$, then $p_3$ is a cv-polynomial respect to $(\sigma'_3, \delta'_3 = 0)$.
\end{enumerate}
\end{proof}

As above, if $(\sigma_3, \delta_3)$ is a quasi-derivation on a ring $R$, $\delta_3$ is said to be {\it algebraic} if there exists a non-zero polynomial  $g \in S = R[x_1;\sigma_1, \delta_1][x_2;\sigma_2,\delta_2][x_3;\sigma_3,\delta_3]$ such that $g(x_1,x_2, \delta_3)=0$. The evaluation of $g(x_1, x_2,x_3)=\sum_{i=0}^n q_i(x_1,x_2)x_3^i$ at $\delta_3$ is defined to be the operator  $g(x_1, x_2, \delta_3) =\sum_{i=0}^n q_i(x_1,x_2)\delta_3^i$ on $S$.

\begin{theorem} \label{teo2.12_paran=3}
Consider Ore extensions $S'= R[x'_1;\sigma'_1, \delta'_1][x'_2;\sigma'_2, \delta'_2][x'_3;\sigma'_3, \delta'_3]$ and $S = R[x_1;\sigma_1, \delta_1][x_2;\sigma_2, \delta_2][x_3;\sigma_3, \delta_3]$. Let $p_1, p_2, p_3$ be elements belonging to $S$ given by $p_1(x_1) = \sum_{i = 0}^k a_{i}x_1^{i}$, $p_2(x_1,x_2) = \sum_{i = 0}^n \sum_{j = 0}^m b_{ij}x_1^{i}x_2^{j}$ and $p_3(x_1,x_2,x_3) = \sum_{i = 0}^t \sum_{j = 0}^s \sum_{l = 0}^w c_{ijl}x_1^{i}x_2^{j}x_3^{l}$. 
\begin{itemize}
    \item [\rm (1)] If $p_3$ is a cv-polynomial respect to the quasi-derivation $(\sigma'_3, \delta'_3)$, then $\delta'_3 = (p_3 - C_{0})(\delta_3) + \delta_{3(C_{0}, \sigma'_3)}$, where $C_{0} = \sum_{i = 0}^t\sum_{j = 0}^sc_{ij0}x_1^{i}x_2^{j}x_3^{0}$, and if also $w \geq 1$, then $c_{ijl}x_1^ix_2^j\sigma_3^l(-) = \sigma'_3(-)c_{ijl}x_1^ix_2^j$ and $\sigma_3\delta_3 = -\delta_3\sigma_3$.
\item [\rm (2)] If $\delta_3$ is not an algebraic derivation, then $p_3$ is a cv-polynomial respect to $(\sigma'_3, \delta'_3)$ if and only if $\delta'_3 = (p_3 - C_{0})(\delta_3) + \delta_{3(C_{0}, \sigma'_3)}$, where the element $C_0$ is given by $\sum_{i = 0}^t\sum_{j = 0}^sc_{ij0}x_1^{i}x_2^{j}x_3^{0}$.
\end{itemize}
\end{theorem}
\begin{proof}
The proof follows the same ideas established in Theorem \ref{teo2.12_paran=2}.
\begin{itemize}
    \item [\rm (1)] Consider the homomorphism $\lambda: S \to \text{End}(R[x_1;\sigma_1,\delta_1][x_2;\sigma_2,\delta_2], +)$ given by $\lambda(x_1) = x_1$, $\lambda(x_2) = x_2$, $\lambda(x_3) = \delta_3$ and for any $r \in R$, $\lambda(r)$ is defined as the left multiplication by $r$ on $R$. Assume that  $p_3$ is a cv-polynomial respect to $(\sigma'_3, \delta'_3)$. Then
\begin{align}\label{equ1.2n=3}
     p_3p_1 &=\phi(x'_3x'_1) = \sigma'_3(p_1)p_3 + \delta'_3(p_1),\quad {\rm and}\\
     p_3p_2 &= \phi(x'_3x'_2) = \sigma'_3(p_2)p_3 + \delta'_3(p_2). \notag
\end{align}
By evaluating the equation (\ref{equ1.2n=3}) respect to the homomorphism $\lambda$, we get
\begin{equation}\label{equation1n=3}
    p_3(x_1,x_2,\delta_3)p_1(x_1)  = \sigma'_3(p_1(x_1))p_3(x_1,x_2, \delta_3) + \delta'_3(p_1(x_1)).
\end{equation}

By evaluating at element 1 for $\delta_3$, since $\delta^j_3(1) = 0$, for $j \geq 1$, we have
\[
p_3(x_1, x_2, \delta_3)p_1(x_1)  = \sigma'_3(p_1(x_1)) \sum_{i = 0}^t\sum_{j = 0}^s c_{ij0}x_1^{i}x_2^{j}x_3^{0}+ \delta'_3(p_1(x_1)).
\]

If $ C_{0} = \sum_{i = 0}^t\sum_{j = 0}^s c_{ij0}x_1^{i}x_2^{j}x_3^{0}$, then
\begin{align*}
   \delta'_3(p_1(x_1)) &=  p_3(x_1,x_2,\delta_3)p_1(x_1) - \sigma'_3(p_1(x_1)) C_{0}\\
   &= p_3(x_1, x_2, \delta_3)p_1(x_1) -C_{0}p_1(x_1) + C_{0}p_1(x_1)- \sigma'_3(p_1(x_1))C_{0}\\
   &= (p_3(x_1,x_2,\delta_3) - C_{0})p_1(x_1) + \delta_{3(C_{0}, \sigma'_3)}(p_1(x_1))\\
   &= (p_2(x_1,x_2,x_3) - C_{0})(\delta_3)(p_1(x_1)) + \delta_{3(C_{0}, \sigma'_3)}(p_1(x_1)),
\end{align*}

whence $\delta'_3 = (p_3 - C_{0})(\delta_3) + \delta_{3(C_{0}, \sigma'_3)}$.

For the next part, consider $w \geq 1$. Since $p_3$ is a cv-polynomial, $p_3p_1 - \sigma'_3(p_1)p_3 = \delta'_3(p_1)$, and so 
    \begin{align}\label{equa212.2n=3}
       &\left(\sum_{i = 0}^t\sum_{j = 0}^s\sum_{l = 0}^w c_{ijl}x_1^{i}x_2^{j}x_3^l\right) \left(\sum_{i = 0}^k a_{i}x_1^{i}\right) - \sigma'_3\left(\sum_{i = 0}^k a_{i}x_1^{i}\right)\sum_{i = 0}^t\sum_{j = 0}^s\sum_{l = 0}^w c_{ijl}x_1^{i}x_2^{j}x_3^l\\
       &= [c_{tsw}x_1^tx_2^s\sigma_3^w(a_kx_1^k) - \sigma'_3(a_kx_1^k)c_{tsw}x_1^tx_2^s]x_3^w +\cdots -  \sigma'_3(p_1)C_{0}\notag\\
       &\ \ + p_{\sigma_3,\delta_3}+  \sum_{i = 0}^t\sum_{j = 0}^t\sum_{l = 0}^w c_{ijl}x_1^{i}x_2^j\delta_3^l(p_1) + C_{0}p_1 \notag\\
       &= \delta'_3(p_1), \notag
    \end{align}
    
where $p_{\sigma_3,\delta_3}$ are the possible combinations between $\sigma_3$ and $\delta_3$. Since (\ref{equa212.2n=3}) is equal to 
    \begin{align*}
        \delta'_3(p_1) &= (p_3 - C_{0})(\delta_3)(p_1) + \delta_{3(C_{0}, \sigma'_3)}(p_1)\\
        &=  \sum_{i = 0}^t\sum_{j = 0}^s\sum_{l = 0}^{w} c_{ijl}x_1^{i}x_2^{j}\delta_3^l(p_1) + C_{0}p_1 - \sigma'_3(p_1)C_{0},
    \end{align*}

we obtain
\[
[c_{tsw}x_1^tx_2^s\sigma_3^w(a_kx_1^k) - \sigma'_3(a_kx_1^k)c_{tsw}x_1^tx_2^s]x_3^w +\cdots + p_{\sigma_3,\delta_3} = 0,
\]
and hence, $c_{ijl}x_1^ix_2^j\sigma_3^l(-) = \sigma'_3(-)c_{ijl}x_1^ix_2^j$  for $1\leq i \leq t$, $1\leq j \leq s$, $1 \leq l \leq w$, and $p_{\sigma_3, \delta_3} = 0$, and so $\sigma_3\delta_3 = -\delta_3\sigma_3$ as desired.

\item [\rm (2)] Let us assume that $\delta_3$ is not an algebraic derivation and let $q(x_1, x_2,x_3) = p_2(x_1, x_2, x_3) - C_{0}$. Then $\delta'_3 = q(x_1, x_2, \delta_3)+ \delta_{3(C_{0}, \sigma'_3)}$, and $q(x_1, x_2, \delta_3) = \delta'_3 - \delta_{3(C_{0}, \sigma'_3)}$ is a $\sigma'_3$-derivation, so for any $r \in R$,
\[
q(x_1,x_2, \delta_3)(p_1(x_1)r) = \sigma'_3(p_1(x_1))q(x_1, x_2, \delta_3)(r) + q(x_1, x_2, \delta_3)(p_1(x_1))r.
\]
If we evaluate by the homomorphism $\lambda$, then 
\[
q(x_1, x_2, \delta_3)\lambda(p_1(x_1)) = \lambda(\sigma'_3(p_1(x_1)))q(x_1,x_2,\delta_3) + \lambda(q(x_1, x_2, \delta_3)(p_1(x_1)))
\]

which holds in the image of $\lambda$. Since $\delta_3$ is not algebraic, $\lambda$ is injective and there exists $\lambda^{-1}$, whence
\[
q(x_1, x_2, x_3)(p_1(x_1)) = \sigma'_3(p_1(x_1))q(x_1,x_2,x_3) + q(x_1, x_2, \delta_3)(p_1(x_1)).
\]
By replacing $q(x_1, x_2, x_3)$ and $q(x_1, x_3,\delta_3)$, 
\begin{align*}
    (p_3(x_1, x_2, x_3) - C_{0})p_1(x_1) &= \sigma'_3(p_1(x_1))(p_3(x_1, x_2,x_3) - C_{0})\\
    &\ \ + (\delta'_3 - \delta_{3(C_{0}, \sigma'_3)})(p_1(x_1))\\
    &= \sigma'_3(p_1(x_1))p_3(x_1, x_2,x_3) - \sigma'_3(p_1(x_1))C_{0} \\
    &\ \ + \delta'_3(p_1(x_1))- \delta_{3(C_{0}, \sigma'_3)} (p_1(x_1))\\
    &= \sigma'_3(p_1(x_1))p_3(x_1, x_2, x_3) - \sigma'_3(p_1(x_1))C_{0}\\
    &\ \ + \delta'_3(p_1(x_1))- C_{0}p_1(x_1) + \sigma'_3(p_1(x_1))C_{0}\\
    &= \sigma'_3(p_1(x_1))p_3(x_1, x_2,x_3)  + \delta'_3(p_1(x_1))\\
    &\ \ - C_{0}p_1(x_1), 
\end{align*}

whence $C_{0} p_1(x_1)$, and $p_3p_1 = \sigma'_3(p_1)p_3 + \delta'_3(p_1)$ as desired. 
\end{itemize}
\end{proof}

\begin{theorem}\label{isoOreextension=3}
Let $\phi$ be a homomorphism of the iterated Ore extension as in Definition \ref{definitionhomooreiteretared3}. $\phi$ is an isomorphism if and only if the associated cv-polynomials have degree one respect to $x_1$, $x_2$ and $x_3$, respectively.
\end{theorem}
\begin{proof}
To show the first implication, it suffices to follow the ideas in the Theorem \ref{isoOreextension} since if we consider $p_3$ as a constant, we get a contradiction. Similarly, if we consider $p_3$ with degree greater than one, a contradiction is also obtained.

Now, we consider the homomorphism $\phi: S' \to S$ defined as $\phi(x'_1) = p_1$, $\phi(x'_2) = p_2$ and $\phi(x'_3) = p_3$, where each of the polynomials $p_i$ has degree one respect to $x_i$, for $i =1,2,3$, respectively. By Theorem \ref{isoOreextension}, the polynomials $p_1 = ax_1+b$ and $p_2= \gamma x_2 + g(x_1)$ satisfy the relation $p_2p_1 = \sigma'_2(p_1)p_2+ \delta'_2(p_1)$, so we suppose that $p_3 = f(x_1,x_2)x_3 + h(x_1,x_2)$. 

Next, we show that $p_1, p_2$ and $p_3$ are cv-polynomials. Consider 
    \begin{align}\label{p3p2ison=3}
        p_3p_2 &= (f(x_1,x_2)x_3 + h(x_1,x_2))(\gamma x_2 + g(x_1))\notag\\
        &= f(x_1,x_2)x_3\gamma x_2 + f(x_1,x_2)x_3 g(x_1) + h(x_1, x_2)\gamma x_2 + h(x_1,x_2)g(x_1) \notag\\
        &= f(x_1,x_2)\sigma_3(\gamma x_2)x_3 + f(x_1,x_2)\delta_3(\gamma x_2) + f(x_1,x_2)\sigma_3(g(x_1))x_3 \notag\\
        &\ \ \ + f(x_1,x_2)\delta_3(g(x_1))+ h(x_1, x_2)\gamma x_2 + h(x_1,x_2)g(x_1), 
    \end{align}
    
and
    \begin{align*}
        \sigma'_3(p_2)p_3 + \delta'_3(p_2) &= \sigma'_3(\gamma x_2 + g(x_1))(f(x_1,x_2)x_3 + h(x_1,x_2)) + \delta'_3(\gamma x_2 + g(x_1))\\
        &= \sigma'_3(\gamma x_2)f(x_1,x_2)x_3 + \sigma'_3(\gamma x_2) h(x_1,x_2) + \sigma'_3(g(x_1))f(x_1,x_2)x_3\\
        &\ \ \ + \sigma'_3(g(x_1))h(x_1,x_2) + \delta'_3(\gamma x_2) + \delta'_3(g(x_1)).
    \end{align*}

These equalities suggest that we should define $\sigma'_3 = (\sigma_{3{f(x_1,x_2)}} \circ \sigma_3)$, that is, $\sigma'_3(-) = f(x_1, x_2)\sigma_3(-)f^{-1}(x_1, x_2)$, whence $f(x_1, x_2) \in R^*$, and we denote it by $f(x_1,x_2)= \beta$. Besides, $\delta'_3 = (\beta \delta_3 + \delta_{3{(h(x_1,x_2),\sigma'_3)}})$. In this way, $p_3 = \beta x_3 + h(x_1, x_2)$, and 
\begin{align*}
        \sigma'_3(p_2)p_3 + \delta'_3(p_2) &= \sigma'_3(\gamma x_2)\beta x_3 + \sigma'_3(\gamma x_2) h(x_1,x_2) + \sigma'_3(g(x_1))\beta x_3\\
        &\ \ \ + \sigma'_3(g(x_1))h(x_1,x_2) + \delta'_3(\gamma x_2) + \delta'_3(g(x_1))\\
        &= \beta \sigma_3(\gamma x_2)\beta^{-1}\beta x_3 + \beta \sigma_3(\gamma x_3)\beta^{-1}h(x_1, x_2) + \beta \sigma_3(g(x_1))\beta^{-1}\beta x_3\\
        &\ \ \ + \beta \sigma_3(g(x_1))\beta^{-1}h(x_1, x_2) + \beta \delta_3(\gamma x_2) + h(x_1,x_2)\gamma x_2 \\
        &\ \ \ - \beta \sigma_3(\gamma x_2)\beta^{-1}h(x_1,x_2) + \beta \delta_3(g(x_1)) + h(x_1,x_2)g(x_1)\\
        &\ \ \ - \beta \sigma_3(g(x_1)) \beta^{-1}h(x_1,x_2)\\
        &= \beta \sigma_3(\gamma x_2)x_3 + \beta \sigma_3(g(x_1))x_3+ \beta \delta_3(\gamma x_2) + h(x_1,x_2)\gamma x_2 \\
        &\ \ \ + \beta \delta_3(g(x_1)) + h(x_1,x_2)g(x_1). 
    \end{align*}
    
By (\ref{p3p2ison=3}), 
    \begin{align*}
        p_3p_2 &= \beta\sigma_3(\gamma x_2)x_3 + \beta\delta_3(\gamma x_2) + \beta\sigma_3(g(x_1))x_3 + \beta\delta_3(g(x_1))\\
        &\ \ \ + h(x_1, x_2)\gamma x_2 + h(x_1,x_2)g(x_1) = \sigma'_3(p_2)p_3 + \delta'_3(p_2). 
    \end{align*}
    
In the same way, by $(\sigma'_3,\delta'_3)$ as it was defined above, it is straightforward to see that $p_3p_1 = \sigma'_3(p_1)p_3 + \delta'_3(p_1)$.

Let us prove that $\phi$ is an isomorphism. We want to show that there exists a homomorphism  $\psi: S \to S'$ such that $\psi \circ \phi = {\rm id}_{S'}$ and $\phi \circ \psi = {\rm id}_{S}$. From Theorem \ref{isoOreextension}, 
\[
\psi(x_1) = a^{-1}x'_1 - a^{-1}\psi(b)\quad {\rm and}\quad \psi(x_2) = \gamma^{-1}x'_2 - \gamma^{-1}g(a^{-1}x'_1-a^{-1}\psi(b)).
\]
 
Let $x'_3 = \psi(\phi(x'_3))$. Then
\begin{align*}
  x'_3 &= \psi(\phi(x'_3)) = \psi(p_3) = \psi(\beta x_3 + h(x_1,x_2)) = \psi(\beta x_3) + \psi(h(x_1,x_2))\\
  &= \beta\psi(x_3) + h(\psi(x_1),\psi(x_2)),  
\end{align*}

whence $\psi(x_3) = \beta^{-1}x'_3 - \beta^{-1}h(\psi(x_1),\psi(x_2))$. These facts show that $\psi \circ \phi = {\rm id}_{S'}$ holds, where $p_1 = ax_1 + b$, $p_2 = \gamma x_2 + g(x_1)$ and $p_3 = \beta x_3 + h(x_1,x_2)$. In an analogous way, we can prove that $\phi \circ \psi = {\rm id}_{S}$.
\end{proof}

\section{$n$-step iterated Ore extensions}\label{Homocasen}

Having in mind the ideas developed in Sections \ref{Homon=2} and \ref{Homon=3}, in this short section we present a general framework of homomorphisms between iterated Ore extensions of $n$-step iterated Ore extensions. 

\begin{definition}\label{definitioniteretaredn}
Consider the Ore extensions $S' =R[x'_1;\sigma'_1,\delta'_1]\cdots[x'_n; \sigma'_n, \delta'_n]$ and  $S = R[x_1;\sigma_1,\delta_1]\cdots[x_n; \sigma_n, \delta_n]$. Let the homomorphism $\phi$ defined as $\phi(x'_i):= p_i$, where $p_i = p(x_1, x_2, \dotsc, x_i)$, with $1 \leq i \leq n$, and such that $p_i$ satisfies for $1 \leq j < i \leq n$ and $r \in R$,
\[
\phi(x'_ix'_j) = \sigma'_i(\phi(x'_j))p_i + \delta'_i(\phi(x'_j)) \quad {\rm and} \quad  \phi(x'_ir) = \sigma'_i(r)p_i + \delta'_i(r),
\]
or equivalently,
\[
p_ip_j = \sigma'_i(p_j)p_i + \delta'_i(p_j)\quad {\rm and}\quad p_i r = \sigma'_i(r)p_i + \delta'_i(r).
\]

We say that the polynomials $p_i$ are the {\em cv-polynomials for iterated Ore extensions respect to} the quasi-derivations $(\sigma'_i, \delta'_i)$ for $1 \leq i \leq n$.

Conversely, if a polynomial $p_i \in S$, where $p_i$ depends on the indeterminates $x_1,x_2, \dotsc, x_i$, with $1 \leq i \leq n$, satisfies the above conditions, where $(\sigma_i', \delta_i')$ are the corresponding quasi-derivations, then we can define $\phi:S' \to S$ by $\phi(x'_i) = p_i$.
\end{definition}

\begin{example}
Consider the single parameter quantum matrix algebra $O_p(M_n(\Bbbk))$ described in Example \ref{examplesOreiterated} (iii). Then
\begin{equation}\label{matrixquantum}
   \small{O_p(M_n(\Bbbk)) \cong \Bbbk[x_{11};\sigma_{11},\delta_{11}]\cdots[x_{1n};\sigma_{1n},\delta_{1n}]\cdots[x_{n1};\sigma_{n1},\delta_{n1}]\cdots[x_{nn};\sigma_{nn}, \delta_{nn}]},
\end{equation} 
with $\sigma_{ij}$ and $\delta_{ij}$, $1 \leq i,j,l,m \leq n$, subject to the relations
\[
\sigma_{ij}(x_{lm}) = \begin{cases}
px_{lm}, & i > l, j = m,\\
px_{lm}, & i = l, j > m,\\
x_{lm}, & i > l, j < m,\\
x_{lm}, & i > l, j > m.
\end{cases}
\]
and $\delta_{ij}(x_{lm}) = (p - p^{-1})x_{im}x_{lj}$ for $i > l, j > m$, otherwise, zero. 

Now, consider $S' = O'_p(M_n(\Bbbk))$ and $S = O_p(M_n(\Bbbk))$ which generating basis $\{x'_{ij}\}$ and $\{x_{ij}\}$, respectively. The idea is to show that if we defined the homomorphism $\phi:S' \to S$ as $\phi(x'_{ij}) := p_{ij} = b_{ji}x_{ji}$, then the polynomials $p_{ij}$ for $1 \leq i,j\leq n$ are cv-polynomials. For $1 \leq i,j, l, m \leq n$, we have
\begin{align*}
    \phi(x'_{ij}x'_{lm}) & = p_{ij}p_{lm} = \sigma'_{ij}(\phi(x'_{lm}))p_{ij} + \delta'_{ij}(\phi(x'_{lm})),\quad {\rm and}\\
   \phi(x'_{ij}r) & = p_{ij}r = \sigma'_{ij}(r)p_{ij} + \delta'_{ij}(r).
\end{align*}

Suppose that $\sigma'_{ij} = \sigma_{ji}$ and $\delta'_{ij} = \delta_{ji}$. Since
\begin{align*}
    p_{ij}p_{lm} &= b_{ji}x_{ji}b_{ml}x_{ml} = \sigma_{ji}(b_{ml}x_{ml})b_{ji}x_{ji} + \delta_{ji}(b_{ml}x_{ml})\\
    &= b_{ji}b_{ml}\sigma_{ji}(x_{ml})x_{ji} + b_{ml}\delta_{ji}(x_{ml}),
\end{align*}

and
\begin{align*}
    \sigma'_{ij}(p_{lm})p_{ij} + \delta'_{ij}(p_{lm})
    &= \sigma_{ji}(b_{ml}x_{ml})b_{ji}x_{ji} + \delta_{ji}(b_{ml}x_{ml})\\
    &= b_{ji}b_{ml}\sigma_{ji}(x_{ml})x_{ji} + b_{ml}\delta_{ji}(x_{ml}),
\end{align*}

it follows that $\phi(x'_{ij}r) = p_{ij}r = \sigma'_{ij}(r)p_{ij} + \delta'_{ij}(r)$ as desired (recall that $\sigma_{ij}$ is the identity and $\delta_{ij}$ is zero for elements $r \in \Bbbk$ for every $1 \leq i,j \leq n$). Therefore, the polynomials $p_{ij} = b_{ji}x_{ji}$ are cv-polynomials respect to the quasi-derivation $(\sigma'_{ij}, \delta'_{ij})$, and $\phi$ is a homomorphism of iterated Ore extensions.
\end{example}

As in the previous sections, for $S$ an iterated Ore extensions with $n$ indeterminates and $p_i \in S$, where $p_i$ is a polynomial in the indeterminates $x_1, x_2$ up to $x_i$, for $1 \leq i \leq n$, we say that $p_i$ is {\it right invariant} if $p_iR[x_1;\sigma_1,\delta_1]\cdots[x_{i-1};\sigma_{i-1}, \delta_{i-1}] \subseteq R[x_1;\sigma_1,\delta_1]\cdots[x_{i-1};\sigma_{i-1}, \delta_{i-1}]p_i$.

\begin{theorem}
Fix $i$ and consider $x_i^n$ a monomial invariant  of degree $n \geq 1$. Let $p_i \in S$ as $p_i = f(x_1, \ldots, x_{i-1})x_i^n + g(x_1, \ldots, x_i)$ with the degree of $g$ less than or equal to $n$. Then:
    \begin{itemize}
        \item [\rm (1)] $p_i$ is a cv-polynomial respect to $(\sigma'_i, \delta'_i)$ if and only if $g(x_1, \ldots, x_i)$ is a cv-polynomial respect to $(\sigma'_i, \delta'_i)$ and $f(x_1, \ldots, x_{i-1})\sigma^{n}_i(-) = \sigma'_i(-)f(x_1, \ldots, x_{i-1})$.
        \item [\rm (2)] If $p_i$ is a left multiple of $f(x_1, \ldots, x_{i-1})$, then $p_i$ is a cv-polynomial if and only if is invariant.
    \end{itemize}
\end{theorem}

\begin{proof}
\begin{itemize}
    \item[\rm (1)] If $p_i$ is cv-polynomial respect to the quasi-derivation $(\sigma'_i, \delta'_i)$, then $p_ip_j = \sigma'_i(p_j)p_i + \delta'_i(p_j)$, and hence
\begin{align}\label{sigmafgn}
        p_ip_j &= \sigma'_i(p_j)[f(x_1, \ldots, x_{i-1})x_i^n + g(x_1, \ldots, x_i)] + \delta'_i(p_j)\notag\\
        &= \sigma'_i(p_j)f(x_1, \ldots, x_{i-1})x_i^n + [\sigma'_i(p_j)g(x_1, \ldots, x_i)] + \delta'_i(p_j).
\end{align}
    
Also, note that
\begin{align}\label{fgsigman}
        p_ip_j &= [f(x_1, \ldots, x_{i-1})x_i^n + g(x_1, \ldots, x_i)]p_j  \notag\\
        &= f(x_1, \ldots, x_{i-1})x_i^np_j + g(x_1, \ldots, x_i)p_j\notag\\
        &= f(x_1, \ldots, x_{i-1})\sigma_i^{n}(p_j)x_i^n + g(x_1, \ldots, x_i)p_j,
\end{align}
    
    where the last expression is obtained due to $x_i^n$ is a monomial invariant of degree $n$. By (\ref{sigmafgn}) and (\ref{fgsigman}), $f(x_1, \ldots, x_{i-1})\sigma^{n}_i = \sigma'_if(x_1, \ldots, x_{i-1})$, and 
\begin{equation}\label{gcvpolynomial}
       g(x_1, \ldots, x_i)p_j =  \sigma'_i(p_j)g(x_1, \ldots, x_i) + \delta'_i(p_j),
\end{equation}
    
which shows that $g(x_1, \ldots, x_i)$ is a cv-polynomial respect to $(\sigma'_i, \delta'_i)$.

Conversely, if $g(x_1, \ldots, x_i)$ is a cv-polynomial respect to the quasi-derivation $(\sigma'_i, \delta'_i)$, then  
    \[
    g(x_1, \ldots, x_i)p_j = \sigma'_i(p_j)g(x_1, \ldots, x_i) + \delta'_i(p_j),
    \]
    and by using that $f(x_1, \ldots, x_{i-1})\sigma^{n}_i(-) = \sigma'_i(-)f(x_1, \ldots, x_{i-1})$, 
    \begin{align*}
        p_ip_j &= f(x_1, \ldots, x_{i-1})x_i^np_j + g(x_1, \ldots, x_i)p_j\\
        &= f(x_1, \ldots, x_{i-1})\sigma_i^{n}(p_j)x_i^n + g(x_1, \ldots, x_i)p_j\\
        &= \sigma'_i(p_j)f(x_1, \ldots, x_{i-1})x_i^n + \sigma'_i(p_j)g(x_1, \ldots, x_i) + \delta'_i(p_j)\\
        &= \sigma'_i(p_j)[f(x_1, \ldots, x_{i-1})x_i^n + g(x_1, \ldots, x_i)] + \delta'_i(p_j)\\
        &= \sigma'_i(p_j)p_i + \delta'_i(p_j).
    \end{align*}
    
    \item [\rm (2)] Suppose that $p_i$ is a cv-polynomial respect to the quasi-derivation $(\sigma'_i, \delta'_i)$. Then $p_ip_j = \sigma'_i(p_j)p_i + \delta'_i(p_j)$, and since $p_i$ is a left multiple of $f(x_1, \ldots, x_{i-1})$, it follows that $g(x_1, \ldots, x_i) = 0$ and $\delta'_i(p_j)= 0$ due to (\ref{gcvpolynomial}). From (1), we know that $f(x_1, \ldots, x_{i-1})\sigma^{n}_i(-) = \sigma'_i(-)f(x_1, \ldots, x_{i-1})$, and
    \[
    p_ip_j =\sigma'_i(p_j)f(x_1, \ldots, x_{i-1})x_i^n= f(x_1, \ldots, x_{i-1})\sigma^{n}_i(p_j)x_i^n, 
    \]
    which shows that $p_i$ is invariant.

Finally, assume that $p_i$ es invariant for each $i$. Then $g(x_1, \ldots, x_i) = 0$, and
\[
p_ip_j = f(x_1, \ldots, x_{i-1})x_i^np_j = f(x_1, \ldots, x_{i-1})\sigma^{n}_i(p_j)x_i^n.
\]
   
If $\sigma'_i$ satisfies $f(x_1, \ldots, x_{i-1})\sigma^{n}_i(-) = \sigma'_i(-)f(x_1, \ldots, x_{i-1})$, then we get that $p_i$ is a cv-polynomial.
\end{itemize}
\end{proof}

For $S = R[x_1;\sigma_1, \delta_1]\cdots[x_n;\sigma_n,\delta_n]$, $\delta_i$ is said to be {\em algebraic} if there exists a non-zero polynomial $g \in S$ such that $g(x_1, x_2, \ldots, \delta_i, \ldots, x_n)=0$. The evaluation of a polynomial $g(x_1,\ldots, x_n)=\sum_{\alpha \in \mathbb{N}^n}a_{\alpha}x_1^{\alpha_1}\cdots x_n^{\alpha_n}$ at $\delta_i$ is defined to be the operator  $g(x_1, x_2, \ldots, \delta_i, \ldots, x_n) = \sum_{\alpha \in \mathbb{N}^n} a_{\alpha}x_1^{\alpha_1}x_2^{\alpha_2}\cdots \delta_i^{\alpha_i}\cdots x_n^{\alpha_n}$ on $S$.

The following theorem shows that if each of the cv-polynomials has a degree greater than one, then these can be characterized with a suitable quasi-derivation $(\sigma'_i, \delta'_i)$. This theorem extends Proposition \ref{LamLeroy1992Theorem2.12}. 

\begin{theorem} \label{teo2.12_paran}
Consider iterated Ore extensions $S'$ and $S$ with $n$ indeterminate. Let $p_i \in S$ given by $p_i = \sum_{\alpha \in \mathbb{N}^i} a_{\alpha}x^{\alpha}$ for $1 \leq i \leq n$. 
\begin{itemize}
    \item [\rm (1)] If $p_i$ is a cv-polynomial respect to the quasi-derivation $(\sigma'_i, \delta'_i)$, then $\delta'_i = (p_i - D_{0})(\delta_i) + \delta_{i(D_{0}, \sigma'_i)}$, where $D_{0} = \sum_{\alpha \in \mathbb{N}^{i-1}} a_{\alpha}x^{\alpha}$ with $\alpha_i = 0$, and if the degree of $x_i$ is greater than or equal to $1$, denote it as $\beta$, then $\sigma_i\delta_i = -\delta_i\sigma_i$ and $a_{\alpha_1\cdots \alpha_{i-1}}x_1^{\alpha_1}\cdots x_{i-1}^{\alpha_{i-1}}\sigma_i^{\alpha_i}(-) = \sigma'_i(-)a_{\alpha_1\cdots\alpha_{i-1}}x_1^{\alpha_1}\cdots x_{i-1}^{\alpha_{i-1}}$.

\item[\rm (2)] If $\delta_i$ is not an algebraic derivation, then $p_i$ is a cv-polynomials respect to $(\sigma'_i, \delta'_i)$ if and only if $\delta'_i = (p_i - D_{0})(\delta_i) + \delta_{i(D_{0}, \sigma'_i)}$, where $D_{0} = \sum_{\alpha \in \mathbb{N}^i} a_{\alpha}x^{\alpha}$ with $\alpha_i = 0$.
\end{itemize}
\end{theorem}
\begin{proof}
\begin{itemize}
\item[\rm (1)] Consider the homomorphism 
$$\lambda_i: S \to {\rm End}(R[x_1;\sigma_1,\delta_1]\cdots[x_i;\sigma_{i-1},\delta_{i-1}], +)$$ 
defined by
\[
\lambda_i(x_{j}) = \begin{cases}
x_j, & j < i,\\
\delta_{i}, & j = i,
\end{cases}
\]    

and for any $r \in R$, $\lambda_i(r)$ is defined by left multiplication on $R$. Since $p_i$ is a cv-polynomial respect to the quasi-derivation $(\sigma'_i, \delta'_i)$, for $j < i$,
\begin{equation}\label{conditionoreiteratedn}
    p_ip_j = \sigma'_i(p_j)p_i + \delta'_i(p_j).
\end{equation}

If we evaluate the homomorphism $\lambda_i$ in (\ref{conditionoreiteratedn}), then
\begin{equation}\label{equationdeltai}
        p_i(x_1,\ldots, \delta_i)p_j = \sigma'_i(p_j)p_i(x_1,\ldots, \delta_i) + \delta'_i(p_j),
\end{equation}

and by evaluating at $1$, we get
\[
p_i(x_1,\ldots, \delta_i)p_j = \sigma'_i(p_j)\sum_{\alpha \in \mathbb{N}^{i-1}} a_{\alpha}x^{\alpha} + \delta'_i(p_j).
\]

Consider $D_0 = \sum_{\alpha \in \mathbb{N}^{i-1}} a_{\alpha}x^{\alpha}$. Then
\begin{align*}
    \delta'_i(p_j) &= p_i(x_1,\ldots, \delta_i)p_j - \sigma'_i(p_j)D_0\\
    &=  p_i(x_1,\ldots, \delta_i)p_j - D_0p_j + D_0p_j - \sigma'_i(p_j)D_0\\
    &= [p_i(x_1,\ldots, \delta_i) - D_0]p_j + \delta'_{i(D_0, \sigma'_i)}(p_j)\\
    &= [p_i(x_1,\ldots, x_i) - D_0](\delta_i)(p_j) + \delta'_{i(D_0, \sigma'_i)}(p_j),
\end{align*}

Now, consider the degree of $x_i$ greater than or equal to $1$. Since $p_i$ is a cv-polynomial, $p_ip_j - \sigma'_i(p_j)p_i = \delta'_i(p_j)$, and so
\begin{align*}
   &\left( \sum_{\alpha \in \mathbb{N}^{i}} a_{\alpha}x^{\alpha}\right)\left( \sum_{\alpha \in \mathbb{N}^{j}} b_{\alpha}x^{\alpha}\right) - \sigma'_i\left( \sum_{\alpha \in \mathbb{N}^{j}} b_{\alpha}x^{\alpha}\right)\sum_{\alpha \in \mathbb{N}^{i}} a_{\alpha}x^{\alpha}\\
   &= [a_{\beta}x^{\alpha_1\cdots\alpha_{i-1}}\sigma^{\beta}_{i}(b_{\gamma}x^{\gamma}) - \sigma^{'}_{i}(b_{\gamma}x^{\gamma})a_{\beta}x^{\alpha_1\cdots\alpha_{i-1}}]x^{\beta} + \cdots - \sigma'_i(p_j)D_0\\
   &\ \ \ + p_{\sigma_{i},\delta_{i}} +  \sum_{\alpha \in \mathbb{N}^i} a_{\alpha}x_1^{\alpha_1}\cdots x_{i-1}^{\alpha_{i-1}}\delta_i^{\alpha_i}(p_j) + D_0p_j = \delta'_i(p_j),
\end{align*}

where $a_{\beta}$ is the leading coefficient of $p_i$ and $p_{\sigma_{i}, \delta_{i}}$ are the possible combinations of $\sigma_{\alpha}$ and $\delta_{\alpha}$. Since $\delta'_i = (p_i - D_{0})(\delta_i) + \delta_{i(D_{0}, \sigma'_i)}$, 
\[
[a_{\beta}x^{\alpha_1\cdots\alpha_{i-1}}\sigma^{\beta}_{i}(b_{\gamma}x^{\gamma}) - \sigma'_{i}(b_{\gamma}x^{\gamma})a_{\beta}x^{\alpha_1\cdots\alpha_{i-1}}]x^{\beta} + \cdots +  p_{\sigma_{i},\delta_{i}} = 0,
\]

and consequently, $\sigma_i\delta_i = -\delta_i\sigma_i$ and 
\[
a_{\alpha_1\cdots \alpha_{i-1}}x_1^{\alpha_1}\cdots x_{i-1}^{\alpha_{i-1}}\sigma_i^{\alpha_i}(-) = \sigma'_i(-)a_{\alpha_1\cdots\alpha_{i-1}}x_1^{\alpha_1}\cdots x_{i-1}^{\alpha_{i-1}}.
\]

\item [\rm (2)] Suppose that $\delta_i$ is not an algebraic derivation. If $q_i = p_i - D_0$ with $q_i$ a polynomial in the indeterminates $x_1, x_2, \dotsc, x_i$, then $\delta'_i = q_i(x_1, \ldots, \delta_i) + \delta_{i(D_0, \sigma'_i)}$. In this way, $q_i(x_1, \ldots, \delta_i) = \delta'_i - \delta_{i(D_0, \sigma'_i)}$ is a $\sigma'_i$-derivation. For any $r \in D$ and $j<i$, 
\[
q_i(p_jr) = \sigma'_i(p_j)q_i(r) + q_i(p_j)r,
\]

and thus, by evaluating the homomorphism $\lambda_i$ and $r = 1$,
\[
q_i(x_1, \ldots, \delta_i)\lambda_i(p_j) = \lambda_i(\sigma'_i(p_j))q_i(x_1, \ldots, \delta_i) + \lambda_i(q_i(x_1, \ldots, \delta_i)(p_j))
\]

which is in the image of $\lambda_i$. Since $\delta_i$ is not algebraic and $\lambda_i$ is injective, then $\lambda^{-1}_i$ exists and 
\[
q_i(x_1, \ldots, x_i)p_j = \sigma'_i(p_j)q_i(x_1, \ldots, x_i) + q_i(x_1, \ldots, \delta_i)(p_j).
\]

By replacing $q_i(x_1, \ldots, \delta_i) = \delta'_i - \delta_{i(D_0, \sigma'_i)}$ and $q_i = p_i - D_0$ in the last equation, we get
\begin{align*}
    (p_i - D_0)p_j &= \sigma'_i(p_j)(p_i - D_0) + (\delta'_i  - \delta_{i(D_0, \sigma'_i)})(p_j)\\
    &= \sigma'_i(p_j)(p_i) - \sigma'_i(p_j)D_0 + \delta'_i(p_j)  - \delta_{i(D_0, \sigma'_i)}(p_j)\\
    &= \sigma'_i(p_j)(p_i) - \sigma'_i(p_j)D_0 + \delta'_i(p_j) + \sigma'_i(p_j)D_0 - D_0p_j\\
    &= \sigma'_i(p_j)(p_i) + \delta'_i(p_j)  - D_0p_j,
\end{align*}

whence $p_ip_j = \sigma'_i(p_i)p_j + \delta'_i(p_i)$, which concludes the proof.
\end{itemize}
\end{proof}

From the discussion above we have immediately the following result.

\begin{proposition}\label{theoisomorphism}
Consider the Ore extensions $S'= R[x'_1;\sigma'_1,\delta'_1]\cdots[x'_n; \sigma'_n, \delta'_n]$ and $S = R[x_1;\sigma_1,\delta_1]\cdots[x_n; \sigma_n, \delta_n]$. The homomorphism $\phi: S' \to S$ defined by $\phi(x'_i) = p(x_i) = x_i$ for each $i = 1, \ldots, n$, is an isomorphism in the sense of cv-polynomials.  
\end{proposition}

\begin{example}
Consider $O_p(M_2(\Bbbk)) \cong \Bbbk[x'_{11}][x'_{12};\sigma'_{12}][x'_{21};\sigma'_{21}][x'_{22};\sigma'_{22},\delta'_{22}]$ and $O_q(M_2(\Bbbk)) \cong \Bbbk[x_{11}][x_{12};\sigma_{12}][x_{21};\sigma_{21}][x_{22};\sigma_{22},\delta_{22}]$ the single parameter quantum matrix algebras generated by $\{ x'_{ij}\}$ and $\{x_{ij}\}$ respectively with $1 \leq i, j \leq 2$. Gaddis \cite[Proposition 3.1]{Gaddis2014} showed that $O_p(M_2(\Bbbk))$ and $O_q(M_2(\Bbbk))$ are isomorphic if and only if $p = q^{\pm 1}$. By using the results obtained in this section, it can be seen that the homomorphism $\phi:O_p(M_2(\Bbbk)) \to O_q(M_2(\Bbbk)) $ given by $\phi(x'_{ij}):= p_{ij} = x_{ij}$ is an isomorphism in the sense of cv-polynomials formulated in Definition \ref{definitioniteretaredn}.
\end{example}

\begin{theorem}
    Let $\phi$ be a homomorphism of the iterated Ore extensions as in Definition \ref{definitioniteretaredn}. $\phi$ is an isomorphism if and only if the associated cv-polynomials $p_i$ for $1\leq i \leq n$ have degree one respect to $x_i$.
\end{theorem}

\begin{proof}
    To prove the first implication, it suffices to follow the ideas of Theorem \ref{isoOreextension} and Theorem \ref{isoOreextension=3}, considering each $p_i$ for $1 \leq i \leq n $ as a constant polynomial or of degree greater than one and thus having a contradiction.

    Conversely, consider the homomorphism $\phi: S' \to S$ defined as $\phi(x'_i) = p_i = a_ix_i + g(x_1, \ldots, x_{i-1})$, where each of the polynomials $p_i$ have degree one respect to $x_i$. Following the same ideas of Theorem \ref{isoOreextension=3}, we have that each $p_i$ satisfies that $p_ip_j = \sigma'_i(p_j)p_i + \delta'_i(p_j)$ for $j < i$. Likewise, to verify that $\phi$ is an isomorphism, consider the homomorphism $\psi: S \to S'$ such that $\psi(x_i) = a_i^{-1}x'_i - a_i^{-1}g(\psi(x_1), \ldots, \psi(x_{i-1}))$. With this homomorphism it is straightforward to verify that $\psi \circ \phi = {\rm id}_{S'}$ and $\phi \circ \psi = {\rm id}_S$, checking that $\phi$ is an isomorphism.
\end{proof}

\begin{remark}
The results obtained in this section can be illustrated with the class of $q$-{\em skew polynomial rings} which have been investigated by some authors (e.g., Goodearl and Letzter \cite{GoodearlLetzter1994} and Haynal \cite{Haynal2008}). 
\end{remark}

\section{Future work}\label{future}

Due to the length of the paper, it is a pending task to investigate the version of Theorem \ref{LamLeroy1992Theorem2.16corrected} for two-step, three-step, and $n$-step iterated Ore extensions. Also, the question on cv-polynomials and homomorphisms by considering the nilpotent elements of the coefficient ring $R$ of the iterated Ore extension, such as in the works of Rimmer, Ferrero and Kishimoto, and Kikumasa, it is also of our interest. On the other hand, as a possible work concerning morphisms between noncommutative rings related with Ore extensions, we consider {\em Double Ore extensions} introduced and characterized by Zhang and Zhang \cite{ZhangZhang2008, ZhangZhang2009}. Our interest is due to Carvalho et al. \cite{Carvalhoetal2011} who described those double Ore extensions that can be presented as two-step iterated Ore extensions. In this way, it is natural to investigate the notions of homomorphism and cv-polynomial in the setting of these double extensions compare them with the results appearing in \cite{Carvalhoetal2011}.

\end{document}